\documentclass[12pt]{article}
\usepackage{geometry}                
\usepackage{graphicx}
\usepackage{amssymb, mathrsfs, enumerate}

\usepackage{amscd}
\usepackage{amssymb}
\usepackage{graphics}
\usepackage{amsmath}
\usepackage[english]{babel}
\usepackage{hyperref}
\usepackage[latin1]{inputenc}

\usepackage{color}      
\usepackage{srcltx}  

\usepackage{amsthm, amssymb}
\usepackage{setspace}
\usepackage[all]{xy}
\usepackage{graphicx}
\usepackage{ifpdf}
\usepackage{babel}
\usepackage{verbatim}

\newtheorem{prop}{Proposition}
\newtheorem{definition}[prop]{Definition}
\newtheorem{defn}[prop]{Definition}
\newtheorem{theorem}[prop]{Theorem}
\newtheorem{thm}[prop]{Theorem}
\newtheorem{cor}[prop]{Corollary}
\newtheorem{lemma}[prop]{Lemma}
\newtheorem{claim}[prop]{Claim}

\newtheorem{rem}[prop]{Remark}
\numberwithin{prop}{section}
\newtheorem{exam}[prop]{Example}
\newtheorem{exer}[prop]{Exercise}
\newtheorem{conj}[prop]{Conjecture}

\newcommand{\RR}{\ensuremath{\mathbb{R}}}

\newcommand{\NN}{\ensuremath{\mathbb{N}}}
\newcommand{\ZZ}{\ensuremath{\mathbb{Z}}}
\newcommand{\QQ}{\ensuremath{\mathbb{Q}}}

\DeclareMathOperator{\Sub}{Sub}
\DeclareMathOperator{\irs}{IRS}

\newcommand{\BZ}{{\mathbb{Z}}}
\newcommand{\BN}{{\mathbb{N}}}
\newcommand{\BR}{{\mathbb{R}}}
\newcommand{\BC}{{\mathbb{C}}}

\newcommand{\act}{\curvearrowright}

\newcommand{\gd}{\delta}

\newcommand{\gC}{\Gamma}
\newcommand{\gc}{\gamma}
\newcommand{\gs}{\sigma}
\newcommand{\gS}{\Sigma}
\newcommand{\gO}{\Omega}

\newcommand{\gep}{\epsilon}

\newcommand{\gL}{\Lambda}
\newcommand{\ga}{\alpha}
\newcommand{\gt}{\tau}

\newcommand{\vol}{\text{vol}}

\newcommand{\SL}{\text{SL}}

\newcommand{\PSL}{\text{PSL}}

\newcommand{\SO}{\text{SO}}

\def\sub{\text{Sub}}
\def\irs{\text{IRS}}

\title{Things we can learn by considering random locally symmetric manifolds}

\author{Tsachik Gelander}

\begin{document}

\maketitle

\begin{abstract}
In recent years various results about locally symmetric manifolds were proven using probabilistic approaches. One of the approaches is to consider random manifolds by associating a probability measure to the space of discrete subgroups of the isometry Lie group. 
The main goals are to prove results about deterministic groups and manifolds by considering appropriate measures. 
In this overview paper, we describe several such results, observing the evolution process of the measures involved. Starting with a result whose proof considered finitely supported measures (more precisely, measures supported on finitely many conjugacy classes) and proceeding with results that were an outcome of the successful and popular theory of IRS (invariant random subgroups). In the last couple of years, the theory has expanded to SRS (stationary random subgroups) allowing to deal with a lot more problems and establish stronger results. In the last section, we shall review a very recent (yet unpublished) result whose proof make use of random subgroups which are not even stationary.

\end{abstract}

{
  \hypersetup{linkcolor=blue}
  \tableofcontents
}

\section{Introduction}

The study of discrete subgroups of Lie groups has played a central role in many parts of mathematics since the middle of the previous century, and has been continuously extremely active. In recent years the theory has been boosted by the development of several probabilistic approaches to study these deterministic highly structured objects. In this overview paper, I will focus on one of these approaches, namely the consideration of random discrete subgroups. Particularly, I will try to describe my personal journey in the development of the theory.

Consider an adjoint semisimple Lie group without compact factors $G$.
By a random discrete subgroup we mean a Borel regular probability measure $m$ on the (metrizable, locally compact) space $\sub_d(G)$ of discrete subgroups of $G$. Let $K\le G$ be a maximal compact subgroup and $X=G/K$ the associated Riemannian symmetric space. Any discrete subgroup $\gL\le G$ corresponds to a locally symmetric orbifold $M=\gL\backslash X$ (which is a manifold if $\Lambda$ is torsion-free). Thus a random discrete subgroup $m$ corresponds to a random $X$-orbifold.

Below we shall describe various theorems about deterministic discrete subgroups and locally symmetric manifolds which are proven by considerations of random subgroups. We wish to demonstrate the development of the theory by showing that the measures $m$ with which we have been working evolved over the years and become more general and sophisticated.  This evolution was required by the type of the results the problems that were considered. 

We start with the result from \cite{ccc}. In that paper A. Levit and the author showed that, counted appropriately, most hyperbolic manifolds are non-arithmetic. More precisely, fixing any dimension $d$ (suppose $d\ge 4$) the number of non-arithmetic complete hyperbolic $d$-manifolds of volume at most $v$ considered up to commensurability is super-exponential as a function of $v$, while the number of arithmetic ones is sub-exponential (see \S \ref{sec:CCC} for more precise statements). The proof relies on a model for constructing random hyperbolic manifolds. However, the probability spaces considered in this work are finite, namely, we analyze a model for constructing random graphs of manifolds which results on finitely many possibilities corresponding to graphs of bounded size. Therefore, one can argue that this is not a probabilistic proof but a combinatorial one. 
The corresponding measures $m$ on $\sub_d(G)$ are not finitely supported but rather supported on finitely many conjugacy classes. This is a very primitive form of the general notion of invariant random subgroups (IRS).

An IRS on $G$ is a probability measure on $\sub(G)$ which is invariant under the $G$-action by conjugation. I believe it is fair to say that the theory of IRS has proven to be much more powerful than the expectation. In particular, IRS in semisimple Lie groups turned out to be an extremely significant tool to the study of lattices and their asymptotic invariants. In \S \ref{sec:IRS} we shall review some of the basics of IRS in semisimple Lie groups. We shall then demonstrate the power of IRS by presenting two theorems about lattices. 
The first is the celebrated theorem of Kazhdan and Margulis establishing the lower bound on the co-volume of lattices. In \S \ref{sec:KM} we will present a surprisingly short proof of this classical theorem using a direct analysis of the space $\text{IRS}(G)$ of invariant random subgroups of $G$. The novelty here is not in the result (although the IRS version is not covered by the original proof of Kazhdan and Margulis) but rather in the simplicity and the straightforwardness of the argument.  

One of the highlights of the theory of IRS is the Stuck--Zimmer rigidity theorem:

\begin{thm}[Stuck--Zimmer]
Let $G$ be a simple Lie group of rank at least $2$. Let $\mu$ be an IRS in $G$ without atoms. Then a $\mu$ random subgroup is almost surely a lattice in $G$.
\end{thm}

The theorem also has a variant for semisimple groups but then the requirement is that $G$ or at least one of its factors has Kazhdan's property $(T)$. The Stuck--Zimmer theorem is a generalization of the normal subgroup theorem of Margulis and as such it is expected that it should hold also for higher rank semisimple groups without property $(T)$, e.g. for $G=\SL_2(\BR)\times\SL_2(\BR)$. This is the famous Stuck--Zimmer conjecture. In \S \ref{sec:SZ} we briefly discuss the Stuck--Zimmer theorem and the conjecture. In \S \ref{sec:rk-1-example} we shall recall the failure of the IRS rigidity in rank one groups by sketching a simple example from \cite{7B}.

In \S \ref{sec:BS} we review the Beniamini--Schramm space $\mathcal{BS}(X)$ associated to a symmetric space $X=G/K$ and its relation with IRS in $G$. The interplay between the geometric point of view naturally associated with $\mathcal{BS}(X)$ and the analytic point of view associated with $\text{IRS}(G)$ is crucial in the proof of the result of the 7 samurai \cite{7A}. In \S \ref{sec:7s} we describe the result of \cite{7A} concerning the asymptotic geometry of higher rank manifolds of finite volume when the volume tends to infinity.
The result is that the injectivity radius cannot be bounded when the volume tends to infinity, and in fact that it grows to infinity at `almost all points'. More precisely, if $\text{rank}(G)\ge 2$ then $\forall R,\gep>0$ there is $V$ such that if $M=\gC\backslash X$ has volume $\ge V$ then $\vol(M_{\ge R})/\vol(M)\ge 1-\gep$, where $M_{\ge R}$ is the $R$-thick part of $M$ where the injectivity radius is $\ge R$. This geometric result has plenty of applications in the theory of analytic invariants. The most famous application is the convergence of normalized Betti numbers (see \cite{7-note, ABBG}), but there are plenty of more refined applications (see \cite{7A}). In \S \ref{sec:Betti} 
we will give the proof of the result concerning the Betti numbers in a special case.

While the notion of invariant random subgroups is quite powerful, the same properties that make it so remarkable are also the ones that demonstrate its limitations. For instance, in view of the Stuck--Zimmer rigidity theorem every non-atomic IRS in $G=\SL_3(\BR)$ is supported on lattices. Therefore IRS are useless in the study of thin subgroups or general discrete subgroups of infinite covolume in $\SL_3(\BR)$. More generally, as semisimple Lie groups are very far from being amenable, focusing on invariant measures is a huge compromise and restricts by much the scope of problems that can be investigated. For this reason, the theory of random subgroups has been expanded in the recent years to non-invariant measures.

Let now $\mu$ be a probability measure on $G$. A probability measure $\nu$ on $\sub(G)$ is called $\mu$-stationary random subgroup (or $\mu$-SRS) if $\nu=\mu*\nu$, where as always $G$ acts on $\sub(G)$ by conjugation. 
The advantage of SRS on IRS is that while it is a much more general notion, a great deal of the theory can be extended to this context. 
Moreover, unlike IRS which are limited to special subgroups, {\it any} closed subgroup $\Lambda$ is related to some SRS. Indeed, one can consider the $\mu$ random walk in $\sub(G)$ starting at $\Lambda$ and hope that some stationary limits retain interesting properties of $\Lambda$. More precisely, consider 
the Cesaro averages
$$
 \nu_n=\frac{1}{n}\sum_{i=1}^n\mu^{(i)}*\gd_\Lambda,
$$
and let $\nu_\infty$ be a weak-$*$ limit. Hopefully by analyzing $\nu_\infty$ one can learn about the initial point $\Lambda$. In order to process this philosophy, one should choose the probability measure $\mu$ carefully.
In \cite{GLM} A. Levit, G.A. Margulis, and myself have worked out a specific measure $\mu=\mu_G$ which possesses some wonderful properties. In particular, as shown in \cite[Theorem 2.2]{FG} for every discrete group $\Lambda$, an SRS $\nu_\infty$ defined as a weak-$*$ limit of Cesaro averages as above is supported on discrete groups. 

In \cite{FG}, M. Fraczyk and I proved a stiffness result for SRS supported on discrete subgroups. This allowed us to prove the following conjecture of Margulis. Recall that a subgroup $\gL\le G$ is called {\it confined} if there is a compact set $C\subset G$ such that 
$$
 C\cap \Lambda^g\setminus\{1\}\ne \emptyset
$$ 
for every $g\in G$.

\begin{thm}\label{thm:confined}
Let $G$ be a simple Lie group of rank at least $2$. A discrete subgroup $\Lambda\le G$ is confined iff it is a lattice.
\end{thm}

Note that this result is new even for subgroups of lattices (e.g. subgroups of $\SL(3,\BZ)$) in which case it establishes a strong version of Margulis' normal subgroup theorem.
 
 Recall that $X=G/K$ denotes the associated symmetric space. 
 Theorem \ref{thm:confined} is equivalent to the following: an $X$-orbifold $M=\gL\backslash X$ of infinite volume
admits injected contractible balls of any radius. In other words, bounded injectivity radius implies finite (and in view of \cite{7A} bounded) volume. 

Section \ref{sec:SRS} is dedicated to an overview of SRS in semisimple Lie groups, and in \S \ref{sec:confined} we give the proof of Theorem \ref{thm:confined}.

In the final section \S \ref{sec:spectral-gap}  I will report on an ongoing joint project with Arie Levit and Uri Bader where we proved in particular the following theorem:
\begin{thm}
Let $G$ be a semisimple Lie group without compact factors and suppose $\text{rank}(G)\ge 2$. Let $\Gamma\le G$ be an irreducible lattice and $\Lambda\le \Gamma$ a confined subgroup. Then $\Lambda$ is of finite index in $\Gamma$.
\end{thm}

The novelty of this theorem is that it holds regardless of property $(T)$. This is a generalization of the normal subgroup theorem that holds in all classical cases. In order to prove this theorem we had to establish a new spectral gap for products. The proof relies on the study of certain random subgroups which are not even stationary. 

\medskip
\noindent
{\bf Acknowledgement:} I wish to thank the referee for several good suggestions and for spotting plenty of typos.

\section{Most hyperbolic manifolds are non-arithmetic}\label{sec:CCC}


A celebrated result of Margulis \cite{Ma2} asserts that finite volume locally symmetric manifolds of rank $>1$ are arithmetic.
Corlette \cite{cor} and Gromov--Schoen \cite{GS} extended this result
to quaternionic and octonionic hyperbolic spaces.
Thus, the only families of finite volume locally symmetric manifolds which may admit 
non-arithmetic members are the real and complex hyperbolic.

A remarkable paper of Gromov and Piatetski--Shapiro \cite{GPS} establishes the existence of a non-arithmetic (real) hyperbolic manifold of finite volume, in any given dimension.

We prove that in fact almost all hyperbolic manifolds are non-arithmetic, with respect to a certain way of counting. Recall that two manifolds are commensurable if they share a common finite cover. 
Fixing the dimension $d>3$, and counting up to commensurability, we show that the number of non-arithmetic hyperbolic $d$-manifolds of volume bounded by $V$ is super-exponential in $V$, 
while the number of arithmetic ones tends to be polynomial.




\subsection*{Arithmetic Groups}

Let us start with some examples:

\begin{exam}
\begin{itemize}
\item $SL(n,\ZZ)$ is a non-uniform lattice in $SL(n,\RR)$.
\item Let $Q(x,y,z,w)=x^2+y^2+z^2-\sqrt{2}w^2$, let $\mathbb{G}=\mathbb{SO}(Q)$ and let $\Gamma=\mathbb{G}(\ZZ[\sqrt{2}])$. Then $\Gamma=\mathbb{H}(\ZZ)$ for a $\QQ$-group $\mathbb{H}$ satisfying 
$$
 \mathbb{H}(\RR)\cong {SO}(3,1)\times SO(4).
$$ 
(The factors correspond to the Galois conjugates of $Q$.)
$\Gamma$ is a lattice in $\mathbb{H}(\RR)$.
Since $SO(4)$ is compact, $\Gamma$ projects to a lattice in $SO(3,1$).
\end{itemize} 
\end{exam}

\begin{definition}
A subgroup $\Gamma\le G$ is called {\it arithmetic} if there is a $\QQ$-algebraic group $\mathbb{H}$ and a surjective map $f:\mathbb{H}(\RR)\to G$ with compact kernel, such that $f(\mathbb{H}(\ZZ))$ 
is commensurable with $\Gamma$.
\end{definition}

One cornerstone of the theory of lattices is the following:
\begin{theorem}[Borel--Harish-Chandra \cite{BoHC}]
Suppose that $G$ is semisimple. Then every arithmetic group is a lattice.
\end{theorem}

A locally symmetric manifold $M=\Gamma\backslash G/K$ associated to an arithmetic group $\Gamma\le G$ is called arithmetic. In particular, arithmetic manifolds have finite Riemannian volume.

Selberg conjectured that under certain conditions the converse of Borel--Harish-Chandra theorem is also true, i.e. that finiteness of volume implies arithmeticity. The formulation of the conjecture was corrected by Piatetski--Shapiro (who properly defined arithmetic groups) and was later proved by Margulis:

\begin{theorem}[Margulis]
Irreducible\footnote{Recall that a locally symmetric manifold is irreducible if it is not commensurable to a direct product of manifolds of smaller dimensions. 
If $G$ is simple all manifolds are locally symmetric.} locally symmetric manifolds associated to a semisimple Lie group of real rank $>1$ are arithmetic.
\end{theorem}

This was extended for certain rank one spaces:

\begin{theorem}[Corlette \cite{cor}, Gromov--Schoen \cite{GS}]
Finite volume locally symmetric manifolds associated to $Sp(n,1)$ and $F_4(-20)$ are arithmetic.
\end{theorem}

The remaining cases are the real and the complex hyperbolic spaces.\footnote{In the complex hyperbolic case there are few examples of non-arithmetic finite volume manifolds in complex dimensions 2 and 3, but the question is wide open for higher dimensions.}


\subsection*{The Gromov--Piatetski-Shapiro construction}

In \cite{GPS} Gromov and Piatetski-Shapiro constructed a non-arithmetic hyperbolic manifold in every dimension. (Prior to that only finitely many examples associated with certain reflection groups were known.) Their idea was to take two arithmetic non-commensurable manifolds $M_1$ and $M_2$ which can be cut along isometric totally geodesic co-dimension one manifolds, cut them, and then glue them by identifying the boundaries of the two. 

Recall that two manifolds $M_1$ and $M_2$ are commensurable if they admit a common finite cover. Two subgroups 
$\Gamma_i,~i=1,2\le G$ are commensurable if their intersection $\Gamma_1\cap\Gamma_2$ has finite index in both.

\medskip

Here is the recipe for the proof of their theorem:

\begin{enumerate}

\item Let $H\le G$ be algebraic groups, and $\Gamma\le G$ an arithmetic subgroup. If $\Gamma\cap H$ is Zariski dense in $H$, then it is an arithmetic lattice in $H$. (This is a consequence of the Borel--Harish-Chandra theorem. For instance, suppose $\Gamma=G(\ZZ)$ and let $\Delta=\Gamma\cap H$. If $\Delta$ is Zariski dense in $H$, then $H$ is defined over $\QQ$ and $\Delta=H(\ZZ)$ is a lattice in $H$.)

\item Let $\Gamma_1,\Gamma_2\le SO(n,1)$ be two arithmetic subgroups. If $\Gamma_1\cap\Gamma_2$ is Zariski dense in $SO(n,1)$ then $\Gamma_1$ and $\Gamma_2$ are commensurable. (Consider $G=SO(n,1)\times SO(n,1)$ and $H$ as the diagonal, take $\Gamma=\Gamma_1\times \Gamma_2$ and apply $(1)$.)

\item Let $M$ be a finite volume complete hyperbolic $n$-manifold and $U\subset M$ a submanifold with a totally geodesic boundary and a non-empty interior. Then $\pi_1(U)$ is Zariski dense in $SO(n,1)$.

\item Let $K$ be a totally real number field. Let $Q_1,Q_2$ be two quadratic forms over $K$ which have signature $(n,1)$ over $\mathbb{R}$ with $n\ge 2$, while all their non-trivial Galois conjugates are positive definite. Let $\Gamma_i$ be the arithmetic group in $SO(n,1)$ corresponding to $Q_i$ and $M_i$ the corresponding hyperbolic orbifold. Then $M_1$ and $M_2$ are commensurable if and only if $Q_1$ is isometric to $\lambda Q_2$ for some scalar $\lambda\in K$ (see \cite[\S 2.6]{GPS}).

\end{enumerate}

\begin{exer}
Using Items $(2),(3)$, show that if $M_1$ and $M_2$ are two non-commensurable arithmetic hyperbolic $n$-manifolds, that admit isometric totally geodesic co-dimension one submanifolds $N_1$ and $N_2$. Then the manifold $M$ obtained by cutting the $M_i$'s along the $N_i's$ and then gluing the two pieces along the boundaries is non-arithmetic.
\end{exer}

Item $(4)$ implies that when $n+1$ is even, the discriminant of the form (valued in $k^*/(k^*)^2$) is invariant under scalar multiplication and hence well defined on a commensurability class.
This is what stands behind the following:

\begin{exam}
1) Suppose that $n+1$ is even, consider the forms:

$$
 Q_1= x_1^2+x_2^2+\ldots+x_n^2-x_{n+1}^2,~\text{and}~Q_2= 2x_1^2+x_2^2+\ldots+x_n^2-x_{n+1}^2
$$
and let $M_1,M_2$ be the corresponding manifolds (take finite covers to get rid of torsion). Then $M_1$ and $M_2$ are non-commensurable and admit 
isometric co-dimension one totally geodesic cuts $N_i$ (along $x_2\equiv 0$). Note that $M_1$ and $M_2$ are non-compact.

\medskip

2) To get a compact example, take 
$$
 R_1= x_1^2+x_2^2+\ldots+x_n^2-\sqrt{2}x_{n+1}^2,~\text{and}~R_2= 3x_1^2+x_2^2+\ldots+x_n^2-\sqrt{2}x_{n+1}^2.
$$
 
\end{exam}

These examples thus provide compact and non-compact finite volume non-arithmetic hyperbolic manifolds in any odd dimension. 
To obtain the even-dimensional examples, Gromov and Piatetski-Shapiro took, inside the above example, a co-dimension $1$ sub-manifold that intersects the $N_i$'s transversally-orthogonally.


\subsection*{Counting hyperbolic manifolds}

Denote by $\rho_n(v)$ the number of complete hyperbolic $n$-manifolds of volume $\le v$, considered up to isometry. 
By Kazhdan--Margulis theorem $\rho_n(v)=0$ for sufficiently small $v$. Teichmuller theory implies that $\rho_2(4\pi)=2^{\aleph_0}$. Thurston showed that $\rho_3(v)=\aleph_0$, for $v\ge v_0$, where $v_0$ is the smallest volume of a complete non-compact hyperbolic $3$ manifold. For $n\ge 4$ however we have

\begin{thm}[Wang \cite{Wang}]
For $n\ge 4$, $\rho_n(v)<\infty$ for all $v>0$.
\end{thm}

Wang's proof is non-effective and gives no estimate on $\rho_n(v)$. Indeed, Wang's beautiful argument basically shows that with respect to a certain topology, the set of conjugacy classes of lattices in $SO(n,1)$ is discrete and the co-volume function is proper. That is, the set of conjugacy classes of lattices of bounded volume is discrete and compact hence finite.

The first effective bound was given by Gromov \cite{Gromov} who showed that, for $n\ge 4$
$$
 \rho_n(v)\le v\exp(\exp(\exp(n+v))).
$$

The correct growth type was determined in my joint work with M. Burger, A. Lubotzky and S. Mozes:

\begin{thm}[\cite{BGLM}]
For $n\ge 4$,
$$
 \log \rho_n(v)\approx v\log v. 
$$
\end{thm}

By '$\approx$' I mean that the ratio between the two sides is bounded below and above by constants.

\begin{rem}
Although in \cite{BGLM} we proved the lower bound only for non-arithmetic manifolds, 
it can be shown that the same growth type applies if we restrict to each of the following four classes: compact-arithmetic, noncompact-arithmetic, compact-nonarithmetic or noncompact-nonarithmetic (see \cite{BGLS}).
\end{rem}

\begin{rem}
The first motivation for these effective estimates came from theoretical physics \cite{Ca1,Ca2}. 
\end{rem}

\subsection*{More refined counting}
Let us say that a complete hyperbolic manifold is minimal if it does not properly cover any other manifold. It is natural to consider the growth of the number of minimal hyperbolic manifolds, i.e. of
$$
 M_n(v):=~\#\{\text{minimal hyperbolic}~n\text{-manifolds of volume}\le v\},
$$
where again we consider manifolds up to isometry.

An even more refined quantity is the number $C_n(v)$ of hyperbolic $n$-manifolds of volume $\le v$, considered up to commensurability. Obviously:
$$
 C_n(v)\le M_n(v)\le \rho_n(v).
$$

Furthermore, denote by $C_n^c(v)$ the number of {\it compact} manifolds, and by $C_n^{nc}(v)$ the number of {\it non-compact} ones (of volume $\le v$, up to commensurability). Then 
$$
 C_n(v)=C_n^c(v)+C_n^{nc}(v).
$$

A first lower bound was given by J. Raimbault\footnote{Raimbault's proof is also probabilistic. The reason that he obtained only exponential lower bound while we obtained super-exponential is that he constructed manifolds moduled over circles while we used general regular graphs. Both proofs were inspired by the construction from \cite{7S,7B} of IRS in $\SO(n,1)$ which is not contained in any lattice.} \cite{Raimbault} who proved  $\log C_n^c(v)\ge \text{Const}\cdot v$. Jointly with A. Levit \cite{ccc} we determined the precise growth:

\begin{thm}[Gelander--Levit 2014, \cite{ccc}]
For $n\ge 4$ we have
$$
 \log C_n^c(v)\approx \log C_n^{nc}(v)\approx v\log v.
$$ 
\end{thm}

Another number that one may wish to consider is that of torsion-free lattices of covolume $\le v$ considered up to quasi-isometries, denoted $QI_n(v)$. However, while all the cocompact lattices in $SO(n,1)$ are quasi-isometric to each other, a beautiful theorem of R. Schwartz \cite{Schwartz} says that two non-co-compact lattices in $SO(n,1)$ are quasi-isometric iff they are, up to conjugacy, commensurable. I.e.
$$
 QI_n(v)=C_n^{nc}(v)+1,
$$
whenever $v$ is at least the minimal volume of a compact hyperbolic $n$-manifold.


\subsection*{Graphs of spaces}

Let me now explain how one constructs plenty of non-commensurable manifolds. The idea is to construct `random' manifolds modeled over regular graphs. 
Starting with two copies of non-commensurable manifolds with isometric boundaries consisting of $4$ components each, 
\begin{center}
\includegraphics[height=4cm]{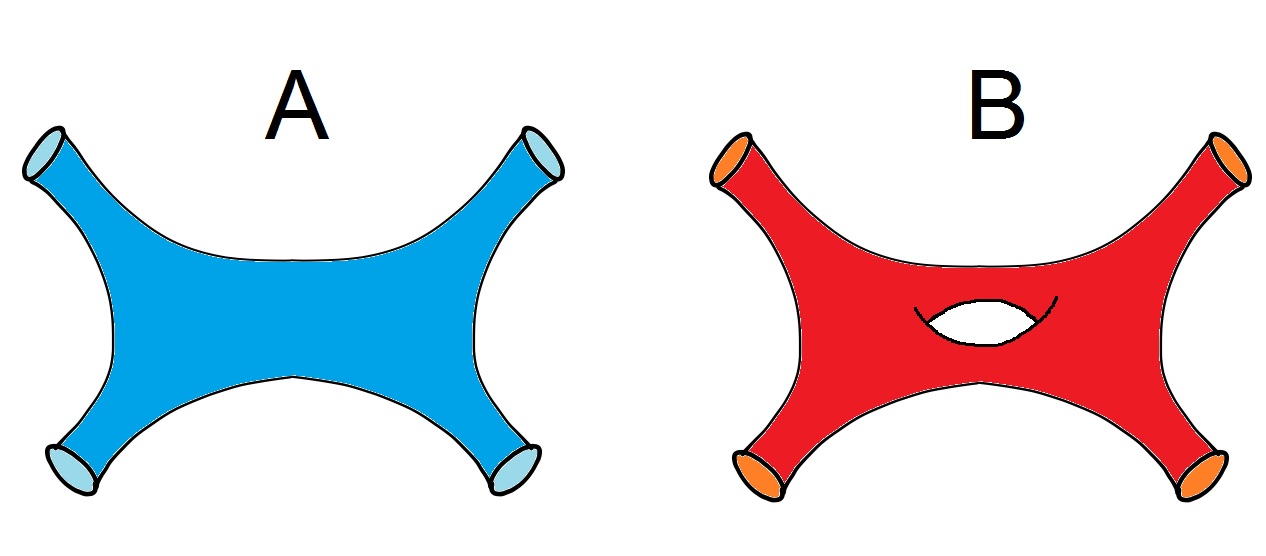}
\end{center} 

for every finite $4$-regular graph, e.g.

\begin{center}
\includegraphics[height=6cm]{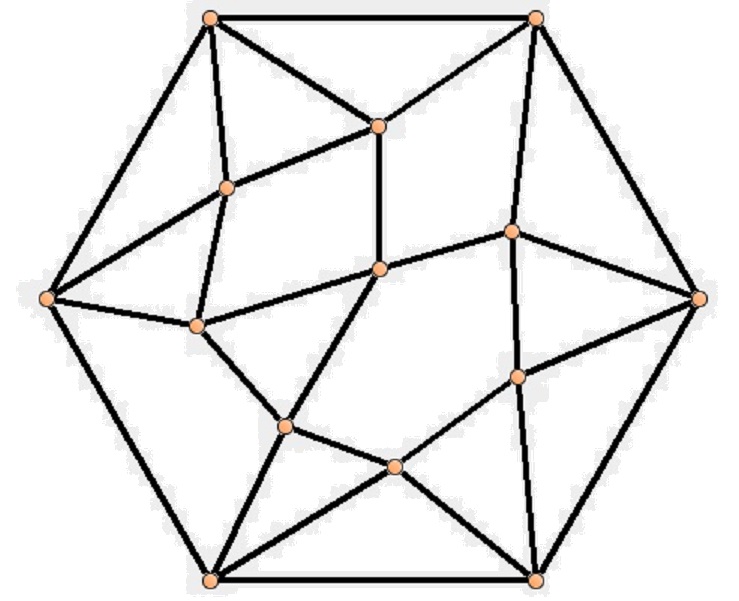}
\end{center}
 
we replace each vertex by a random copy of either $A$ or $B$ above

\begin{center}
\includegraphics[height=6cm]{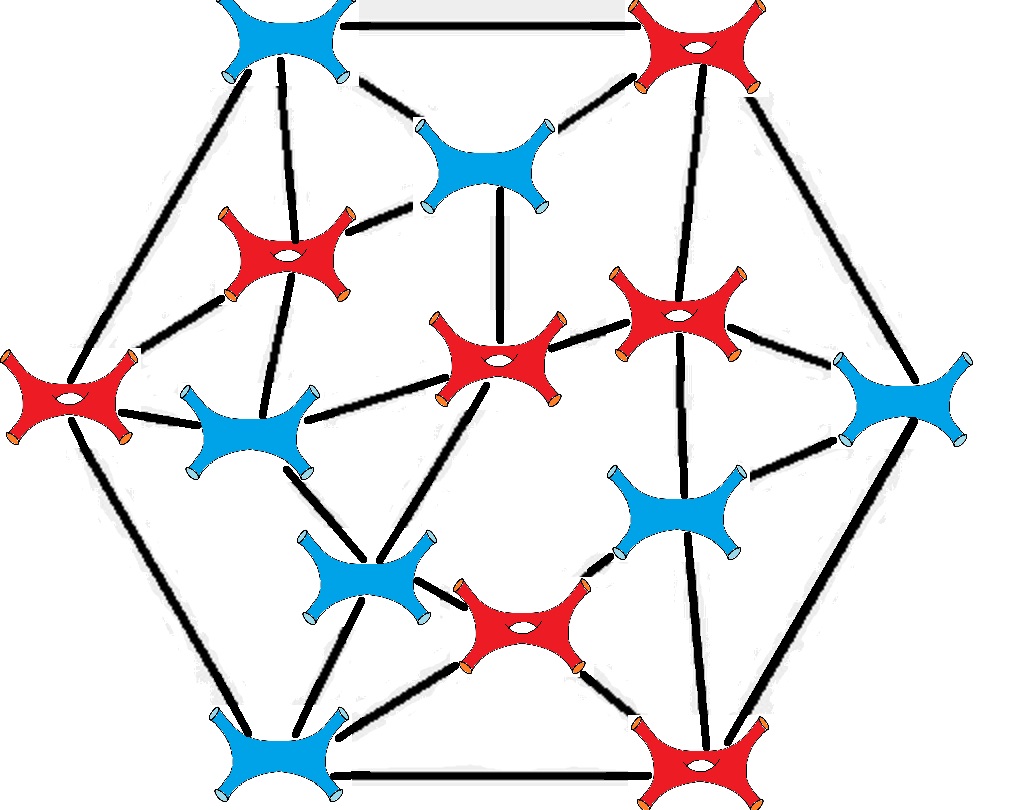}
\end{center} 
 and glue boundary components following the edges pattern.

By a simple counting argument, the logarithm of the number of $4$-regular graphs with $v$ vertices is $\approx v\log v$.\footnote{$4$-regular graphs can be thought of as (uncolored) Schrier graphs of $F_2$. The latter are in bijection with transitive actions on the set of $v$ elements, i.e. with maps $F_2\to \text{Sym}(v)$. Ignoring the transitivity and the coloring issues, there are $v!^2$ such maps.}
 
However, they are all commensurable. By Leighton's theorem \cite{Le}, two colored finite graphs are commensurable iff they share the same universal covering colored 
 tree. 
 
Using elementary asymptotic group theory, we are able to construct sufficiently many non-commensurable colored graphs. However, it is not clear that manifolds constructed over non-commensurable graphs cannot be commensurable as manifolds. Let me mention a geometric lemma that helps us dealing with this issue:

\begin{lemma}
\label{lemma:twononcommensurablearedisjoint}
Let $X_1$ and $X_2$ be two submanifolds with totally geodesic boundaries and finite volume inside two non-commensurable arithmetic hyperbolic $n$-manifolds. Assume either that $\partial X_1$ and $\partial X_2$ are compact and $n \ge 3$, or that $\partial X_1$ and $\partial X_2$ have finite volume and $n \ge 4$.

If $W$ is any complete hyperbolic manifold with two embedded submanifolds $U_1, U_2 \hookrightarrow W$ that admit finite isometric covers $p_i : U_i \to X_i$ for $i=1,2$, then the intersection $U_1 \cap U_2$ has an empty interior.
\end{lemma}

Making use of the Borel density theorem (lattices are Zariski dense) and the Borel--Harish-Chandra theorem (see Item (1) in the recipe of the proof of GPS above) the proof of the last lemma is reduced to the following geometric fact which is of independent interest:

\begin{prop}
\label{prop:proper_submanifold_has_finite_volume}
Let $M$ be an $n$-dimensional complete finite volume hyperbolic manifold without boundary and let $N$ be a properly embedded totally geodesic $k$-dimensional sub-manifold with $1<k\le n$. Then $N$ has finite volume.
\end{prop}

\subsection*{The building blocks}

Some technical issues lead us to use $6$ building blocks, instead of $2$. Moreover, we produce one admissible parcel of compact building blocks, and another one consisting of non-compact building blocks, establishing the desired lower bounds on $C_n^c(v)$ as well as on $C_n^{nc}(v)$.

The harder case is the even dimension as the discriminant is useless. Instead, we address the Legendre and the Hilbert symbols. In the non-compact case, we may then use
$$
 Q_p=px_1^2+x_2^2+\ldots+x_n^2-2x_{n+1}^2,~\text{with}~p=5,13,29,37,53,61.
$$
In the compact case we use
$$
 R_p=px_1^2+x_2^2+\ldots+x_n^2-\sqrt{2}x_{n+1}^2,~\text{with}~p=17,41,97,137,193,241.
$$
The proof that the associated manifolds are non-commensurable relies on the following:

\begin{thm}[Gauss]
For prime $p\equiv 1(mod~4)$ the equation $x^4=2$ has an integer solution modulo $p$ iff $p=x^2+64y^2$ for some $x,y\in\NN$.
\end{thm}


\section{Invariant Random Subgroups}\label{sec:IRS}

The results described in the previous section were established by considering the uniform measure on the finite set of $4$-regular graphs of a given size. This corresponds to a very primitive example of an IRS in $\SO(n,1)$. We will now consider general IRS, which we will define for every locally compact second countable group. For a more comprehensive exposition of IRS we refer to the lecture notes \cite{G} (or \cite{ICM}).

Consider a locally compact second countable group $G$. We denote by $\Sub(G)$ the Chabauty space of closed subgroups of $G$. The Chabauty topology on $\Sub(G)$ is metrizable and can be defined as follows. Let $d$ be a proper metric on $G$ which is compatible with the topology, and define the distance $\rho(H_1,H_2)$ between two closed subgroups $H_1,H_2$ by:
$$
 \rho(H_1,H_2):=\int_0^\infty \text{Hd}(H_1\cap B(1,r),H_2\cap B(1,r))e^{-r}dr,
$$
where $B(1,r)$ is the closed $r$-ball around the identity in $G$ and $\text{Hd}$ is the Hausdorff distance defined on compact sets 
$$
 \text{Hd}(A_1,A_2)=\max \{d(x,A_i): i=1,2~\text{and}~x\in A_{3-i}\}.
$$ 
Thus $H_1$ and $H_2$ are close to each other if their intersections with a large ball in $G$ are close in the Hausdorff distance. 

The space $\Sub(G)$ is always compact but in general it is rather complicated. However, it is a canonical $G$ space as $G$ acts on it by conjugation. 

\begin{defn}
An IRS is a $G$-invariant Borel regular probability measure on $\Sub(G)$. 
\end{defn}

Basic examples of IRS come from normal subgroups and lattices:
\begin{enumerate}
\item 
A Dirac IRS, $\delta_N$ corresponds to a closed normal subgroup $N\lhd G$. 
\item
Let $\Gamma\le G$ be a lattice and let $\mu$ be the $G$-invariant probability measure on $G/\Gamma$. Consider the map 
$ G/\Gamma\to\Sub(G),~g\Gamma\mapsto g\Gamma g^{-1}$ and set $\mu_\Gamma$ as the push-forward of $\mu$. Then $\mu_\Gamma$ is an IRS associated with the conjugacy class of $\Gamma$ and a $\mu_\Gamma$-random subgroup is a conjugate of $\gC$.
\end{enumerate}

For instance let $\Sigma$ be a closed hyperbolic surface and normalize its Riemannian measure. Every unit tangent vector yields an embedding of $\pi_1(\Sigma)$ in $PSL_2(\RR)$. Thus the probability measure on the unit tangent bundle corresponds to an IRS of Type (2) above.  

\begin{center}
\includegraphics[height=4cm]{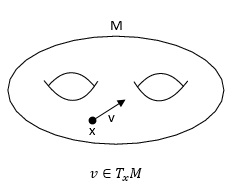}
\end{center}

Remarkably, the study of IRS plays an important role in the development of the modern theory of lattices and their asymptotic invariants. We shall give several illustrations for this phenomenon in subsequent sections. 

IRS are strongly related to p.m.p actions. Indeed, given a probability space $(X,m)$ and a measure preserving $G$-action $G\curvearrowright X$, it follows by Varadarajan compact model \cite{Vara} that the stabilizer of almost every point in $X$ is closed in $G$. Thus the push-forward of $m$ under the stabilizer map $x\mapsto G_x$ is an IRS in $G$. By the following theorem every IRS arises in this way:

\begin{thm}\label{thm:satbilizers}\cite[Theorem 2.6]{7A}
Let $G$ be a locally compact group and $\mu$ an IRS in $G$. Then there is a probability space $(X,m)$ and a measure preserving action $G\curvearrowright X$ such that $\mu$ is the push-forward of the stabilizer map $X\to\text{Sub}(G)$.
\end{thm}

This was proved in \cite{AGV} for discrete groups and in \cite[Theorem 2.4]{7S} for general $G$. The first thing that comes to mind is to take the given $G$ action on $(\text{Sub}(G),\mu)$, but then the stabilizer of a point $H\in\text{Sub}(G)$ is the normalizer $N_G(H)$ of $H$ rather than $H$. To correct this one considers the larger space $\text{Cos}_G$ of all cosets of all closed subgroups, as a measurable $G$-bundle over $\text{Sub}(G)$. Defining an appropriate invariant measure on $\text{Cos}_G\times\mathbb{R}$ and replacing each fiber by a Poisson process on it, gives the desired probability space (see \cite[\S 2]{7S} for details).

\subsection*{IRS in semisimple groups}

Let us now restrict to the case where $G$ is a semisimple Lie group. Treating IRS as a generalization of lattices, a natural thing to do is to establish IRS analogs for theorems about lattices. An important example is the Borel density theorem.

\begin{thm}[Borel density theorem for IRS \cite{7A,GL}]\label{thm:BDT}
Let $G$ be a connected non-compact simple (center-free) Lie group. Let $\mu$ be an IRS on $G$ without atoms. Then a random subgroup is $\mu$-a.s. discrete and Zariski dense.  
 \end{thm}

\begin{proof}[Sketching a proof]
By the Von-Neumann--Cartan theorem, every closed subgroup of $G$ is a Lie group. 
Let $\text{Gr}(\text{Lie}(G))$ denote the Grassmanian variety of subspaces of the Lie algebra $\text{Lie}(G)$.
Consider the map $\psi:\Sub(G)\to\text{Gr}(\text{Lie}(G))$, $H\mapsto \text{Lie}(H)$. It is upper semi-continuous, hence measurable. Let $\nu=\psi_*(\mu)$. Then $\nu$ is $\text{Ad}(G)$-invariant and hence, by the Furstenberg lemma \cite{Fu}, is supported on $\{\text{Lie}(G),\{0\}\}$. Finally observe that $\psi^{-1}(\text{Lie}(G))=G$ and $\psi^{-1}(\{0\})$ is the set of discrete subgroups of $G$. Assuming that $G$ is not an atom of $\mu$ the result about discreteness follows. 

Arguing similarely with respect to the map $H\mapsto \text{Lie}(\overline{H}^Z)$ one obtains the result about Zariski density. For this part one has to show that an IRS in $G$ cannot be supported on finite groups (see \cite[Lemma 2.10]{7A}).
\end{proof}

We refer to \cite{7A,GL} for more details and more general versions including the case of semisimple groups over general local fields.  

Let us denote by $\Sub_d(G)$ the space of discrete subgroups of $G$. 
 By Cartan's theorem, a closed subgroup of $G$ is a Lie group. Since the dimension is a semicontinuous function, the set
 $\Sub_d(G)$ of zero-dimensional subgroups is an open subset of the compact space $\Sub(G)$. We shall say that an IRS $\mu$ is discrete if $\mu(\Sub_d(G))=1$, and let $\text{IRS}_d(G)$ denote the space of discrete IRS of $G$.
Since $G$ is isolated in $\Sub(G)$ (see \cite{Za}, see \cite{G}) we obtain the following useful fact:

\begin{cor}\label{cor:compact}
Let $G$ be a connected non-compact simple Lie group. Then $\text{IRS}_d(G)$ is compact. 
\end{cor}

\begin{rem}
For the simplicity of the exposition we supposed above that $G$ is simple rather than semisimple. However, both results can be extended to semisimple Lie groups without compact factors. 
We refer to \cite[Theorem 1.9]{GL} for the Borel density theorem (it is proved there for a wider class of semisimple analytic groups over local fields) and to \cite[Proposition 2.2]{WUD} for the analog of Corollary \ref{cor:compact} in that context.
\end{rem}


\section{Kazhdan Margulis theorem for IRS}\label{sec:KM}

Recall the celebrated theorem of Kazhdan and Margulis:

\begin{thm}\cite{KM}\label{thm:KM}
Let $G$ be a connected semisimple Lie group without compact factors. There is an identity neighborhood $U\subset G$ such that for every discrete subgroup $\Lambda\le G$ there is some $g\in G$ for which 
$$
 g\Lambda g^{-1}\cap U=\{1\}.
$$
\end{thm}

The equivalent geometric formulation of this theorem is:

\begin{thm}
Given a symmetric space of non-compact type $X$, there is $\epsilon>0$ such that for every complete $X$-orbifold $M=\Lambda\backslash X$ there is a point where the injectivity radius is at least $\epsilon$. In other words, the $\epsilon$-thick part of $M$ is non-empty.
\end{thm} 

The immediate remarkable consequence is that there is a positive lower bound on the volume of $X$-orbifolds, or equivalently, on the co-volume of lattices in $G$.

\begin{figure}[h]
    \centering
    \includegraphics[width=0.8\textwidth]{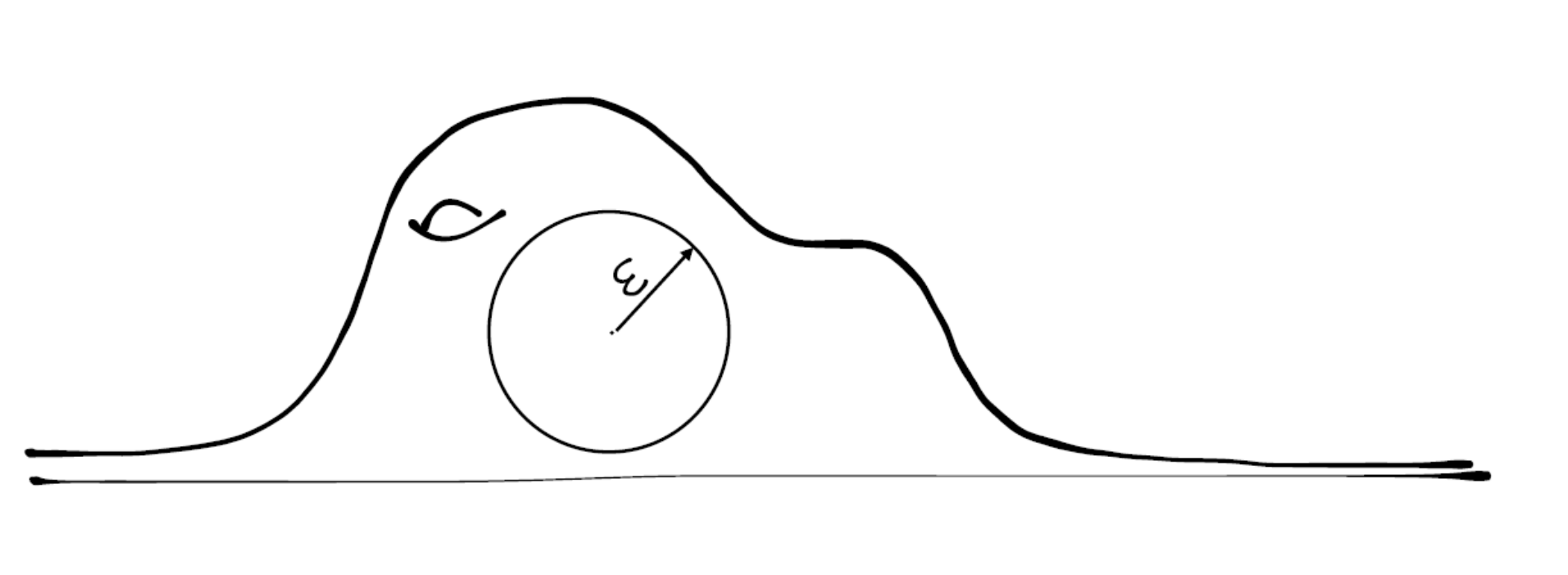}
    \caption{Every $X$-manifold has a thick part.}
    \label{fig:awesome_image}
\end{figure}

There are by now three different proofs of Theorem \ref{thm:KM}. The original proof \cite{KM}, a proof by Gromov \cite{BGS} which is more geometric and the proof from \cite{WUD} which relies on IRS.  Each of these proofs has some advantages on the others. For example, the original proof applies to all discrete subgroups (not just lattices) while the proof of Gromov applies to general Hadamard spaces (rather than symmetric spaces). The proof from \cite{WUD} does not apply for general discrete groups but it does apply for general IRS for which it gives a stronger statement than the original theorem, even when restricting to lattices. However, the main advantage of that proof is its simplicity. 
It happens occasionally in mathematics that a consideration of more general notions allows softer and less technical treatments. This depends however on formulating appropriate definitions. 

Recall that a family of subgroups $\mathcal{F}\subset\sub(G)$ is called {\it uniformly discrete} if there is an identity neighborhood $U\subset G$ such that $\gC\cap U=\{1\}$ for all $\gC\in \mathcal{F}$.

\begin{defn}
A family $\mathcal{F}\subset\irs_d(G)$ of invariant random subgroups is said to be {\it weakly uniformly discrete} if for every $\gep>0$ there is an identity neighbourhood $U_\gep\subset G$ such that
$$
 \mu (\{\gC\in\sub (G):\gC\cap U_\gep\ne\{1\}\})<\gep
$$
for every $\mu\in \mathcal{F}$.
\end{defn}


\begin{thm}\label{thm:KM-IRS}
Let $G$ be a connected non-compact semisimple Lie group. Then $\irs_d(G)$ is weakly uniformly discrete.
\end{thm}
By restricting to the family $\mu_\gC$ of IRS associated to lattices $\gC\le G$ and picking for instance $\gep=1/2$ one recovers the original Kazhdan--Margulis theorem. Indeed, for every lattice $\gC$ there is $g\in G$ such that $g\gL g^{-1}\cap U_{\frac{1}{2}}=\{1\}$. Moreover, ``for at least $1/2$" of the choices of $g$ this will hold, i.e.  
$$
m(\{g\gC\in G/\gC: g\gC g^{-1}=\{1\}\})\ge 1/2
$$ 
where $m$ is the $G$-invariant probability measure on $G/\gC$.

\medskip
\noindent
{\it Proof of Theorem \ref{thm:KM-IRS}.}
Let $U_n, n\in\BN$ be a descending sequence of compact sets in $G$ which form a base of identity neighbourhoods, and set
$$
  K_n=\{\gC\in\sub_G:\gC\cap U_n=\{1 \}\}.
$$

Since $G$ has NSS (no small subgroups), i.e. there is an identity neighbourhood which contains no non-trivial subgroups, we have:

\begin{lemma}\label{K_n-open}
The sets $K_n$ are open in $\sub(G)$.
\end{lemma}

\begin{proof}
Fix $n$ and
let $V\subset U_n$ be an open identity neighbourhood which contains no non-trivial subgroups, such that ${V^2}\subset U_n$. It follows that a subgroup $\gC$ intersects $U_n$ non-trivially iff it intersects $U_n\setminus V$. Since $U_n\setminus V$ is compact, the lemma is proved.  
\end{proof}

In addition, observe that the ascending union $\bigcup_n K_n$ exhausts $\sub_d(G)$, the set of all discrete subgroups of $G$. Therefore we have:

\begin{claim}\label{clm}
For every $\mu\in\irs_d(G)$ and $\gep>0$ we have $\mu(K_n)>1-\gep$ for some $n$.\qed
\end{claim}

Let
$$
 \mathcal{K}_{n,\gep}:= \{\mu\in\irs_d(G):\mu(K_n)>1-\gep\}.
$$
Since $\sub(G)$ is metrizable, it follows from Lemma \ref{K_n-open} that $\mathcal{K}_{n,\gep}$ is open. 
By Claim \ref{clm}, for any given $\gep>0$, the sets $\mathcal{K}_{n,\gep},~n\in\BN$ form an ascending cover of $\irs_d(G)$. Since the latter is compact, we have $\irs_d(G)\subset \mathcal{K}_{m,\gep}$ for some $m=m(\gep)$. 
It follows that 
$$
 \mu\big(\{ \gC\in\sub(G):\gC~\text{intersects}~U_m~\text{trivially}\}\big)>1-\gep,
$$
for every $\mu\in\irs_d(G)$. This completes the proof of Theorem \ref{thm:KM-IRS}.
\qed



\section{The Stuck--Zimmer theorem}\label{sec:SZ}

One of the most striking results about IRS is the Stuck--Zimmer theorem \cite{SZ} which was proved almost two decades before the study of IRS has become popular. 
Recall that a p.m.p action of a semisimple Lie group $G$ is {\it irreducible} if every factor of $G$ acts ergodically.

\begin{thm}\label{thm:SZ}
Let $G$ be a semisimple Lie group of rank at least two and with Kazhdan's property $(T)$. Then every irreducible p.m.p action of $G$ is essentially free or transitive. 
\end{thm}

As follows from the Borel density theorem\footnote{Recall that the Borel density theorem says that if $G/H$ carries an invariant probability measure, then $H$ is Zariski dense. This imples that the connected component $H^\circ$ is normal in $G$ and hence a product of simple factors. When the action is non-trivial and irreducible, this implies that $H^\circ=\{1\}$.} every non-trivial transitive irreducible p.m.p action of a semisimple Lie group $G$ is of the form $G\act G/\Gamma$ where $\Gamma\le G$ is an irreducible lattice. Thus, in view of Theorem \ref{thm:satbilizers}, Theorem \ref{thm:SZ} can be reformulated as follows:

\begin{thm}\label{thm:SZ2}
For $G$ as in Theorem \ref{thm:SZ}, every non-trivial irreducible IRS is of the form $\mu_\Gamma$ where $\Gamma\le G$ is an irreducible lattice.
\end{thm}

The Stuck--Zimmer theorem is an ergodic counterpart of the celebrated normal subgroup theorem of Margulis. In order to see why it imply the NST one can argue as follows. Let $G$ be as in Theorem \ref{thm:SZ}, let $\gC\le G$ be an irreducible lattice and let $N\lhd \gC$ be a non-trivial normal subgroup of $\gC$. Consider the map 
$$
 G/\gC\to \Sub(G),~g\gC\mapsto gNg^{-1},
$$
and let $\mu$ be the push-forward of the $G$-invariant probability measure on $G/\gC$. Then $\mu$ is a non-trivial irreducible IRS. Hence by Theorem \ref{thm:SZ2} it is supported on lattices in $G$. It follows that $N$ is a lattice in $G$ and hence of finite index in $\gC$.

Like the NST of Margulis, the Stuck--Zimmer theorem is a consequence of the tension between amenability and property $(T)$. However, while Margulis was able to prove the NST regardless of property $(T)$ \cite{Ma1}, the ergodic counterpart (the Stuck--Zimmer Theorem) is still unknown under that generality. 

\begin{conj}[The Stuck--Zimmer conjecture]\label{conj:SZ}
The analog of Theorem \ref{thm:SZ} holds regardless of property $(T)$. That is every irreducible p.m.p action of a semisimple Lie group of rank at least $2$ is essentially free or transitive. 
\end{conj}

We remark that if at least one of the factors of $G$ has property $(T)$ then the result holds, as was confirmed by Hartman and Tamuz \cite{HT}. 
But without this assumption, the conjecture is wide open. The first case to consider is $G=\SL_2(\RR)\times\SL_2(\RR)$. What Harman and Tamuz realized is that the argument of Stuck and Zimmer implies in general that in higher rank every non-trivial irreducible IRS is supported on co-amenable subgroups. 

\begin{defn}
Let $G$ be a locally compact group. A discrete subgroup $\Lambda\le G$ is {\it co-amenable} if the Hilbert space $L^2(G/\Lambda)$ admits $G$ asymptotically invariant vectors. 
\end{defn}

\begin{thm}[\cite{SZ,HT}]\label{thm:SZ,HT}
Let $G$ be a semisimple Lie group of real rank at least $2$. Then every non-trivial irreducible IRS in $G$ is supported on co-ameanable subgroups.
\end{thm}

Thus, the following general conjecture implies the Stuck--Zimmer conjecture \ref{conj:SZ}:

\begin{conj}
Let $G$ be a semisimple Lie group without compact factors and rank at least two. Let $\Lambda\le G$ be a discrete co-amenable subgroup which projects densely to every proper factor of $G$. Then $\Lambda$ is a lattice in $G$.
\end{conj}

In \S \ref{sec:spectral-gap} below I will report about a recent progress in this direction.

\begin{rem}
Let us recall that the Stuck--Zimmer theorem relies heavily on the Nevo-Zimmer intermediate factor theorem which was actually established later \cite{NZ4}. Later on Arie Levit extended the Nevo--Zimmer and the Stuck--Zimmer theorem for semisimple analytic groups over non-archimedean local fields \cite{Levit}.

\end{rem}


\section{An exotic IRS in rank one}\label{sec:rk-1-example}

In the lack of Margulis' normal subgroup theorem, there are IRS supported on non-lattices. Indeed, if $G$ has a lattice with an infinite index normal subgroup $N\lhd\gC$, arguing as in the previous page, one obtains an ergodic p.m.p. space for which almost any stabilizer  is a conjugate of $N$. 

We shall now give a more interesting example, in the lack of rigidity:

\begin{exam}[An exotic IRS in $\PSL_2(\BR)$, \cite{7S}]

Let $A,B$ be two copies of a surface with 2 open discs removed equipped with distinguishable hyperbolic metrics, such that all the 4 boundary components are geodesic circles of the same length.

\begin{center}
\includegraphics[height=4cm]{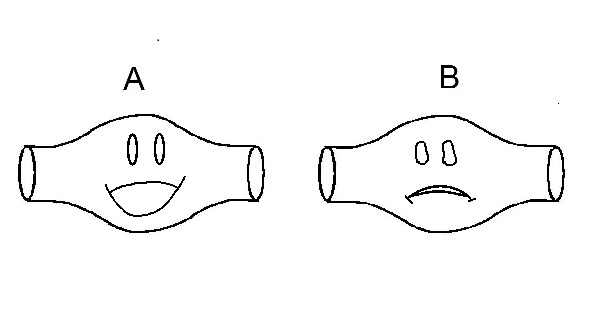}
\end{center}

 

Now consider the space $\{A,B\}^\BZ$ with the Bernoulli measure $(\frac{1}{2},{1\over 2})^\BZ$. Any element $\ga$ in this space is a two sided infinite sequence of $A$'s and $B$'s and we can glue copies of $A,B$ `along a bi-infinite line' following this sequence. This produces a random surface $M^\ga$. Choosing a probability measure on the unit tangent bundle of $A$ (resp. of $B$) we define an IRS in $\PSL_2(\BR)$ as follows. First choose $M^\ga$ randomly, next choose a point and a unit tangent vector in the copy of $A$ or $B$ which lies at the place $M^\ga_0$ (above $0$ in the sequence $\ga$), then take the fundamental group of $M^\ga$ according to the chosen point and direction. 
 
\begin{center}
\includegraphics[height=2cm]{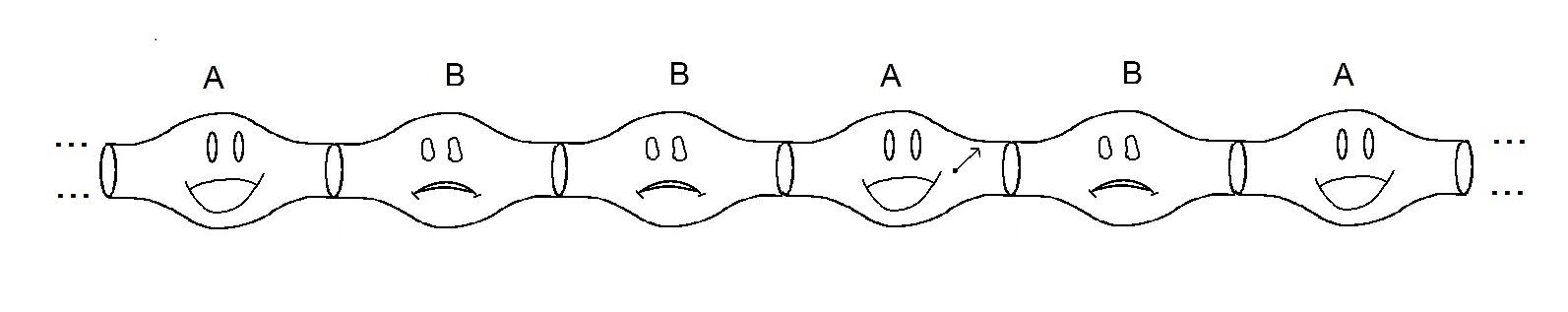}
\end{center} 

It can be shown that almost surely the corresponding group is not contained in a lattice in $\PSL_2(\BR)$.
Analog constructions can be made in $\SO(n,1)$ for all $n$'s (see \cite[Section 13]{7S}).
We refer to \cite{7B} for more examples of exotic IRS in rank one.

\end{exam}


\section{The Benjamini--Schramm, the weak-* and the IRS topologies}\label{sec:BS}

\subsection*{The Gromov--Hausdorff topology}

Given a compact metric space $X$, the Hausdorff distance $\text{Hd}_X(A,B)$ between two closed subsets is defined as
$$
 \text{Hd}_X(A,B):=\inf\{\gep: N_\gep(A)\supset B~\text{and}~N_\gep(B)\supset A\},
$$
where $N_\gep(A)$ is the $\gep$-neighborhood of $A$. The space $2^X$ of closed subsets of $X$ equipped with the Hausdorff metric, is compact.

Given two compact metric spaces $X,Y$, the Gromov distance $\text{Gd}(X,Y)$ is defined as
$$
 \text{Gd}(X,Y):=\inf_Z\{\text{Hd}_Z(i(X),j(Y))\},
$$ 
where the infimum is taken
over all compact metric spaces $Z$ admitting isometric copies $i(X),j(Y)$ of $X,Y$ respectively. 
If $(X,p),(Y,q)$ are pointed compact metric spaces, i.e. ones with a chosen point, we define the Gromov distance 
$$
 \text{Gd}((X,p),(Y,q)):=\inf_Z\{\text{Hd}_Z(i(X),j(Y))+d_Z(i(p),j(q))\}.
$$ 

The Gromov--Hausdorff distance between two pointed proper (not necessarily bounded) metric spaces $(X,p),(Y,q)$ can be defined as
$$
 \text{GHd}((X,p),(Y,q)):=\sum_{n\in\BN} {1\over 2^n}\text{Gd}((B_X(n),p),(B_Y(n),q)),
$$ 
where $B_X(n)$ (resp. $B_Y(n)$) is the ball of radius $n$ around $p$ in $X$ (resp. around $q$ in $Y$).

\subsection*{The Benjamini--Schramm topology}

Let $\mathcal{M}$ be the space of all (isometry classes of) pointed proper metric spaces equipped with the Gromov--Hausdorff topology. This is a huge space and for many applications, it is enough to consider compact subspaces of it obtained by bounding the geometry. That is, let $f(\gep,r)$ be an integer-valued function defined on $(0,1)\times\BR^{>0}$, and let $\mathcal{M}_f$ consist of those spaces for which $\forall \gep,r$, the $\gep$-entropy\footnote{The maximal cardinality of an $\gep$-discrete set.} of the $r$-ball 
$B_X(r,p)$ around the special point is bounded by $f(\gep,r)$, i.e. no $f(\gep,r)+1$ points in $B_X(r,p)$ form an $\gep$-discrete set. Then $\mathcal{M}_f$ is a compact subspace of $\mathcal{M}$.   

In many situations, one prefers to consider some variants of $\mathcal{M}$ which carry more information about the spaces. 
For instance when considering graphs, it may be useful to add colors and orientations to the edges. The Gromov--Hausdorff distance defined on these objects should take into account the coloring and orientation.
Another example is smooth Riemannian manifolds, in which case it is better to consider framed manifolds, i.e. manifold with a chosen point and a chosen frame at the tangent space at that point. In that case, one replaces the Gromov--Hausdorff topology with the ones determined by $(\gep,r)$ relations (see \cite[Section 3]{7S} for details), which remembers also the directions from the special point.

We define the {\it Benjamini--Schramm space} $\mathcal{BS}=\text{Prob}(\mathcal{M})$ to be the space of all Borel probability measures on $\mathcal{M}$ equipped with the weak-$*$ topology. Given $f$ as above, we set $\mathcal{BS}_f:=\text{Prob}(\mathcal{M}_f)$. Note that $\mathcal{BS}_f$ is compact.

The name of the space is chosen to hint that this is the same topology induced by `local convergence', introduced by Benjamini and Schramm in \cite{BS}, when restricting to measures on rooted graphs. Recall that a sequence of random rooted bounded degree graphs converges to a limiting distribution iff for every $n$ the statistics of the $n$ ball around the root (i.e. the probability vector corresponding to the finitely many possibilities for $n$-balls) converges to the limit. 

The case of general proper metric spaces can be described similarly. A sequence $\mu_n\in\mathcal{BS}_f$ converges to a limit $\mu$ iff for any compact pointed `test-space' $M\in\mathcal{M}$, any $r$ and arbitrarily small\footnote{This doesn't mean that it happens for all $\gep$.} $\gep>0$, the $\mu_n$ probability that the $r$ ball around the special point is `$\gep$-close' to $M$ tends to the $\mu$-probability of the same event.

\begin{exam}
An example of a point in $\mathcal{BS}$ is a measured metric space, i.e. a metric space with a Borel probability measure. 
A particular case is a finite volume Riemannian manifold --- in which case we scale the Riemannian measure to be one, and then randomly choose a point (and a frame).
\end{exam}

Thus a finite volume locally symmetric space $M=\gC\backslash G/K$ produces both a point in the Benjamini--Schramm space and an IRS in $G$. This is a special case of a general analogy that we will now describe. Given a symmetric space $X$, let us denote by $\mathcal{M}(X)$ the space of all pointed (or framed) complete Riemannian orbifolds whose universal cover is $X$, and by
$\mathcal{BS}(X)=\text{Prob}(\mathcal{M}(X))$ the corresponding subspace of the Benjamini--Schramm space. 

Let $G$ be a semisimple Lie group without compact factors, with maximal compact subgroup $K\le G$ and an associated Riemannian symmetric space $X=G/K$. There is a natural map 
$$
 \{\text{discrete subgroups of }~G\}\to \mathcal{M}(X),~\gC\mapsto \gC\backslash X.
$$  
It can be shown that this map is continuous, hence inducing a continuous map
$$
 \irs_d(G)\to \mathcal{BS}(X).
$$
It can be shown that the later map is one-to-one, and since $\irs_d(G)$ is compact, it is
a homeomorphism to its image (see \cite[Corollary 3.4]{7S}). One can also characterize the image of this map which is the space of probability measures on $\mathcal{M}(X)$ which are invariant under the geodesic flow (see \cite{AB}).

\section{A result of the 7 Samurai}\label{sec:7s}

The following result from \cite{7S,7A} can be interpreted as `large higher rank manifolds are almost everywhere fat':

\begin{thm}\label{thm:7-main}
Let $X$ be an irreducible symmetric space of rank at least $2$.
Let $M_n=\gC_n\backslash X$ be a sequence of finite volume $X$-orbifolds with $\vol(M_n)\to\infty$. Then $M_n\to X$ in the Benjamini--Schramm topology.
\end{thm} 

This means that for any $r$ and $\gep$ there is $V(r,\gep)$ such that if $M$ is an $X$-manifold of volume $v\ge V(r,\gep)$ then $\frac{\vol(M_{\ge r})}{v}\ge 1-\gep$, where $M_{\ge r}$ denotes the $r$-thick part of $M$, i.e. the set of points in $M$ where the injectivity radius is at least $r$.

\begin{figure}[h]
    \centering
    \includegraphics[width=0.8\textwidth]{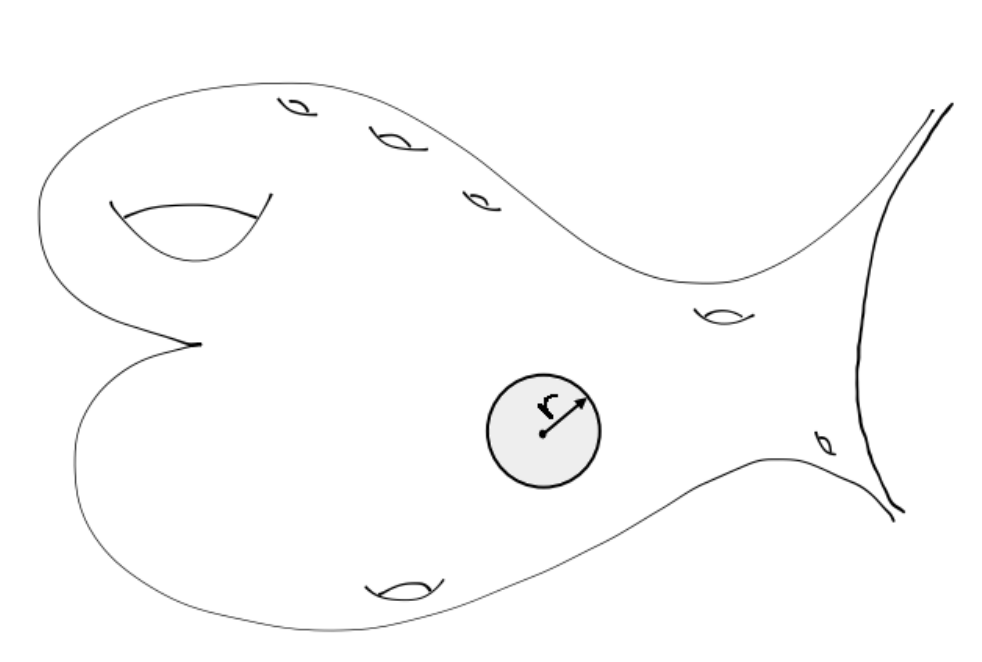}
    \caption{A large volume manifold is almost everywhere fat.}
    \label{whale}
\end{figure}

Using the dictionary from the previous section we may reformulate Theorem \ref{thm:7-main} in the language of IRS:

\begin{thm}[\cite{7S,7A}]\label{thm:main-IRS}
Let $G$ be a simple Lie group of rank at least $2$.
Let $\gC_n\le G$ be a sequence of lattices with $\vol(G/\gC_n)\to\infty$ and denote by $\mu_n$ the corresponding IRS. Then $\mu_n\to \gd_{\{1\}}$.
\end{thm} 

The proof makes use of the equivalence between the two formulations. 

Recall that by Kazhdan's theorem, $G$ has property (T). This implies that a limit of ergodic measures is ergodic:

\begin{thm}[\cite{GW}]\label{thm:GW}
Let $G$ be a group with property (T) acting by homeomorphisms on a compact Hausdorff space $X$. Then the set of ergodic $G$-invariant probability Borel measures on $X$ is $w^*$-closed.
\end{thm}   

The idea behind Theorem \ref{thm:GW} is that if $\mu_n$ are probability measures converging to a limit $\mu$ and $\mu$ is not ergodic, then there is an $L^2(\mu)$ function on $X$ which is $G$-invariant, orthogonal to the constants and with norm $1$. Approximating it one gets 
continuous function $f$ on $X$ which, as a function in $L_2(\mu)$ is almost $G$-invariant, orthogonal to the constants, and with norm $1$. Thus for large $n$ we have that $f$ is almost invariant in $L_2(\mu_n)$, almost orthogonal to the constants, and with norm almost $1$. Since $G$ has property (T) it follows that there is an invariant $L_2(\mu_n)$ function close to $f$, so $\mu_n$ cannot be ergodic.  

\medskip
\noindent
{\it Sketch of proof of Theorem \ref{thm:main-IRS}.}
Let $M_n=\gC_n\backslash X$ be $X$-orbifolds of finite volume with $\vol(M_n)\to\infty$ as in \ref{thm:7-main}, let $\mu_n$ be the corresponding IRS and let $\mu$ be a weak-$^*$ limit of $\mu_n$. Our aim is to show that $\mu=\gd_{\{1\}}$. Up to replacing $\mu_n$ by a subnet, we may suppose that $\mu_n\to \mu$. By Theorem \ref{thm:GW} we know that $\mu$ is ergodic. However, in view of Theorem \ref{thm:SZ2},
the only ergodic IRS on $G$ are $\gd_G,\gd_{\{1\}}$ and $\mu_\gC$ for $\gC\le G$ a lattice. 

Thus, in order to prove Theorem \ref{thm:main-IRS} (equivalently \ref{thm:7-main}) we have to exclude the cases $\mu=\gd_G$ and $\mu=\mu_\gC$. The case $\mu=\gd_G$ is impossible since $G$ is an isolated point in $\text{Sub}(G)$ (see \cite{Za} or \cite{G}). Let us now suppose that $\mu=\mu_\gC$ for some lattice $\gC\le G$ and aim towards a contradiction. For this, we will adopt the formulation of Theorem \ref{thm:7-main} again. Thus we suppose that $M_n\to M=\gC \backslash X$ in the Benjamini--Schramm topology.

Recall that Property (T) of $G$ implies that there is a lower bound $C>0$ for the Cheeger constant of all finite volume $X$-orbifolds. For our purposes, the Cheeger constant of an orbifold $M$ can be defined as the infimum
$$
 C=\inf_S \frac{\vol (N_1(S))}{\min \{\vol(M'),\vol(M'')\}},
$$
where $S$ is any subset which disconnects the manifold to two sides $M',M''$ and $N_1(S)$ is its $1$-neighbourhood.\footnote{The sides $M',M''$ are not assumed to be connected but any path from $M'$ to $M''$ must cross $S$.}

Since $M=\gC\backslash X$ has finite volume, for an arbitrary point $p\in M$ if we pick $R$ large enough then
$$
 \vol(B_M(p,R-1))>(1-C)\vol(M),
$$
where $B_M(p,R-1)$, that is the $(R-1)$-ball around $p$ in $M$. 

In particular, taking 
$S=\{x\in M:d(x,p)=R\}$ and $M'=B_M(p,R)$ we have that:
$$
  \frac{\vol(N_1(S))}{\vol(M')}\le\frac{\vol(N_1(S))}{\vol(B_M(p,R-1))}\le\frac{\vol(B_{M}(p,R+1)\setminus B_{M}(p,R-1))}{\vol(B_{M}(p,R-1))}<C.
$$

This in itself does not contradict property (T) since the complement $M''=M\setminus M'$ could (and actually has to) be very small (or even empty if $M$ is compact). 

Now since $M_n$ converges to $M$ in the BS-topology, it follows that for $n$ sufficiently large, there must be a point $p_n\in M_n$, whose $R+1$ neighborhood $B_{M_n}(p_n,R+1)$ is arbitrarily close (in the Gromov--Housdorff sense) to $B_M(p,R+1)$. 
Setting 
$$
 S_n=\{x\in M_n:d(x,p_n)=R\},~M_n':=\{x\in M_n:d(x,p_n)<R\}
$$ 
we get that
$$
 \frac{\vol(N_1(S_n))}{\vol(M_n')}<C.
$$ 
Bearing in mind that $\vol(M_n)\to\infty$, we get that for large $n$, the complement $M_n'':=M_n\setminus M_n'$ has arbitrarily large volume, and in particular $\vol(M_n')\le\vol(M_n'')$.
Now, this contradicts the assumption that $C$ is the Cheeger constant of $X$.
See Figure \ref{fig:Cheeger}.
\qed


\begin{figure}[h]\label{fig:Cheeger}
    \centering
    \includegraphics[width=0.6\textwidth]{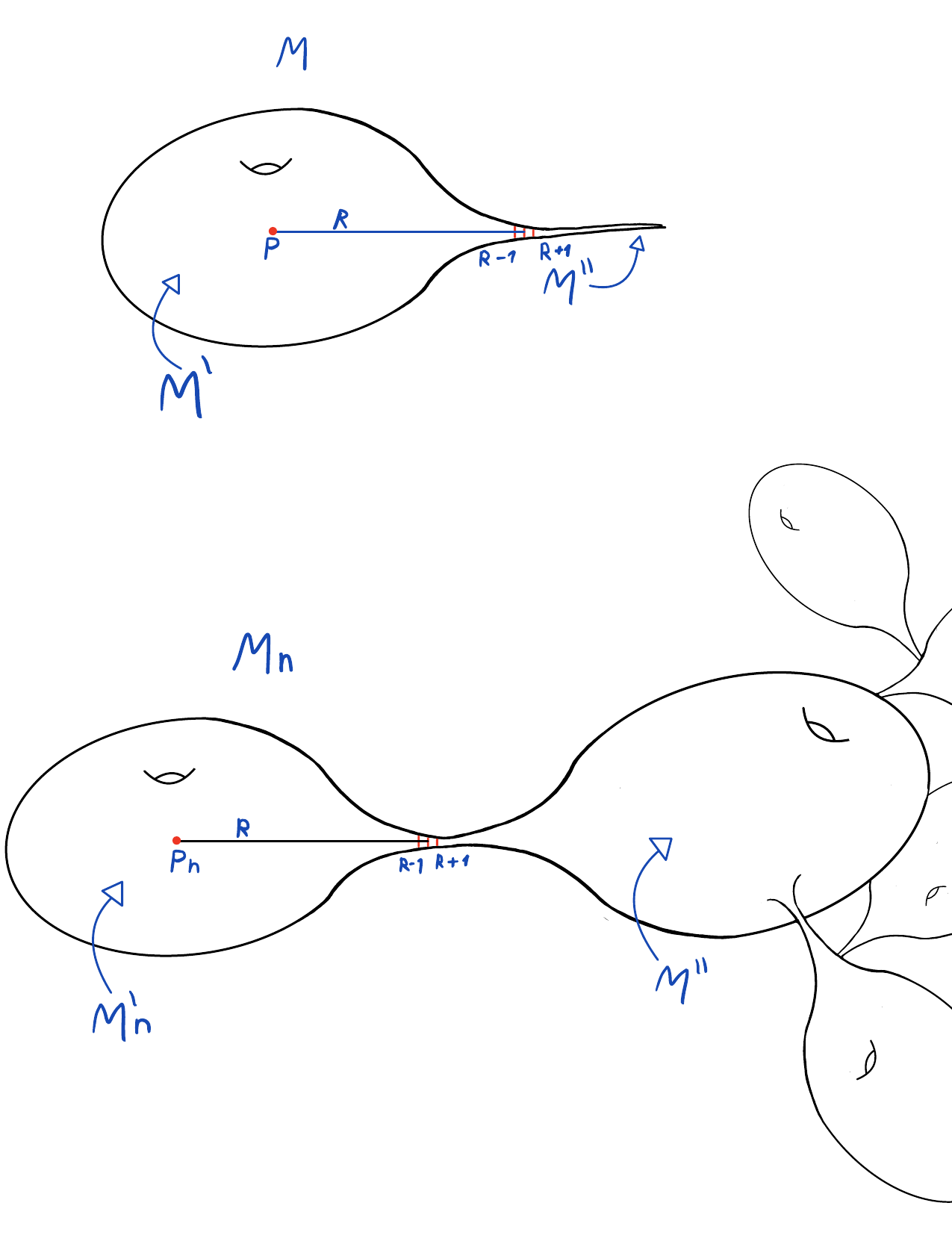}
    \caption{The Cheeger constant of $M_n$ is too small.}
    \label{whale}
\end{figure}


\vspace{1cm}
\noindent
{\it An alternative proof.}
Instead of relying on Kazhdan's property (T) one could make use of local rigidity. If $\mu_n\to \mu$ where $\mu$ is the IRS corresponding to some lattice $\gC$ then there must be points in the support of $\mu_n$ which converge to points in the support of $\mu$. However, $\mu_n$ is supported on the conjugacy class of $\gC_n$ while $\mu$ is supported on the conjugacy class of $\gC$. It follows that there must be conjugates of $\gC_n$ which converges to $\gC$ in the Chabauty topology. However, by Margulis' super-rigidity theorem
$\gC$ is locally rigid and hence also Chabauty locally rigid. It follows that for large $n$, $\gC_n$ must contain a conjugate of $\gC$. This however implies that $\vol(G/\gC_n)\le\vol(G/\gC)$ in contrast to the assumption that $\vol(G/\gC_n)\to \infty$. This approach was taken in \cite{GL} where the work \cite{7A} has been generalized to semisimple groups over non-archimedean local fields. We refer to \cite{GL} for more details.
\qed

\medskip

As explained in the `alternative proof' above, the argument using the Cheeger constant could be replaced by local rigidity. However, property (T) is crucial in two other parts of the proof. Indeed, the proof relies on Theorem \ref{thm:GW} as well as the Stuck--Zimmer theorem \ref{thm:SZ} (which in turn relies on property (T)).

\begin{rem}
In \cite{7S,7A} Theorem \ref{thm:7-main} was proved in the greater generality of irreducible lattices in higher rank semisimple groups with property (T). In fact it is enough to suppose that $G$ is higher rank and at least one of the factors has property (T). In \cite{Levit1} A. Levit proved the analog theorem for higher rank groups regardless of property (T) but assuming that the lattices $\gC_n$ are congruence (instead of Kazhdan's property (T) he made use of Clozel's property $\tau$, see\cite{Cl}). If Serre's conjecture is true, then every higher rank irreducible lattice is congruence. 
\end{rem}

\begin{rem}
Obviously the analog of Theorem \ref{thm:7-main} is false for $G=\SO(n,1)$. For example, one can start with a compact hyperbolic manifold $M=\gC\backslash \mathcal{H}^n$ and take a descending sequence of finite index normal subgroups $N_i\lhd\gC$ such that the infinite intersection $\cap N_i$ is non-trivial and consider the covering manifolds $M_i=N_i\backslash  \mathcal{H}^n$. Also, the examples constructed in \S \ref{sec:CCC} have uniformly bounded injectivity radius. However, maybe for some `more rigid'  rank one groups (for instance, this is interesting for $\text{Sp}(n,1)$) the analog of Theorem \ref{thm:7-main} holds for some mysterious reason.   
\end{rem}

\begin{figure}[h]
    \centering
    \includegraphics[width=0.6\textwidth]{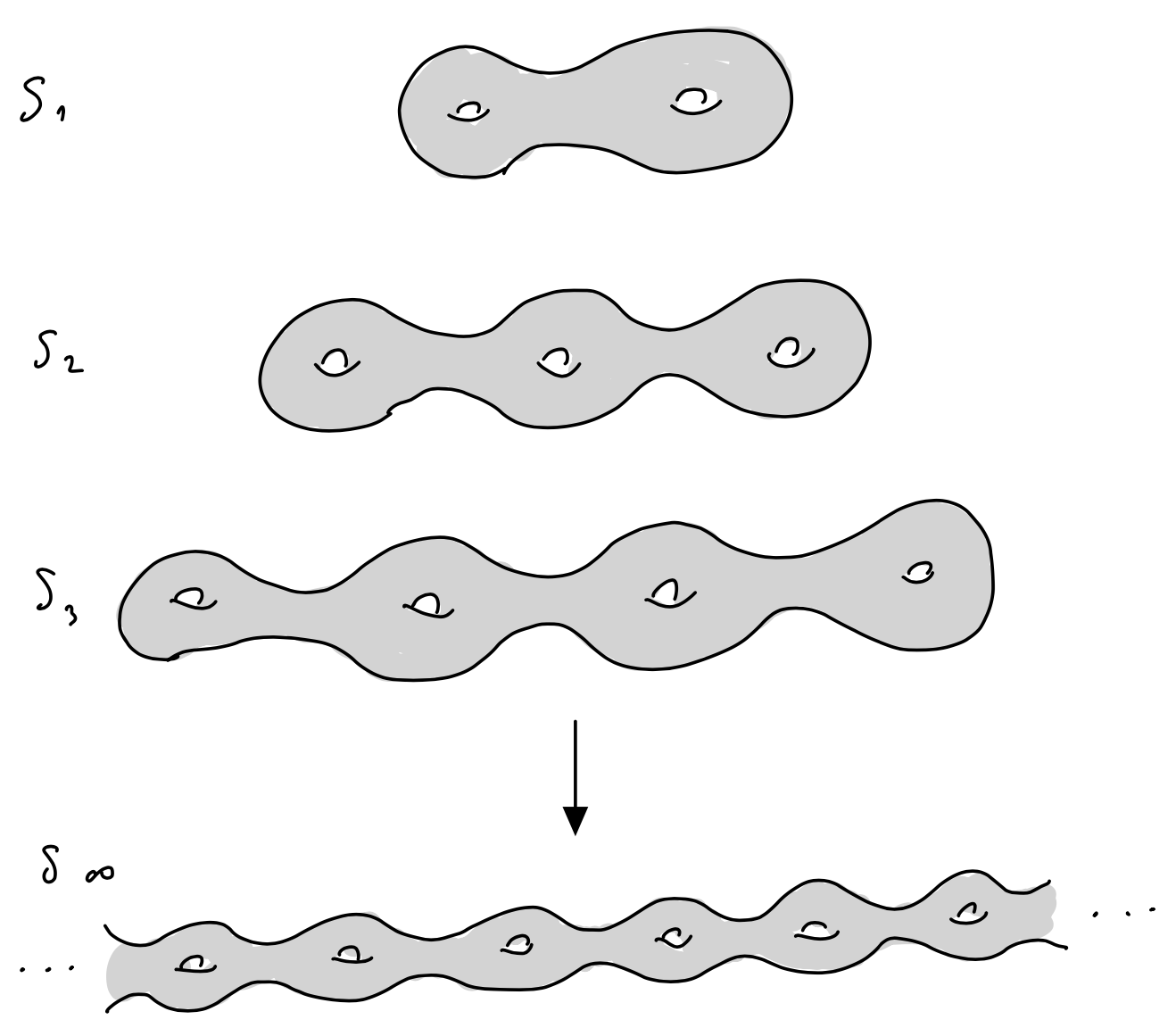}
    \caption{Surfaces with bounded injectivity radius.}
    \label{whale}
\end{figure}


\section{Convergence of normalized Betti numbers}\label{sec:Betti}

The most famous application of Theorem \ref{thm:7-main} is the convergence of normalized Betti numbers. This was first proved in \cite{7-note} for sequences of compact manifolds with injectivity radius bounded below (i.e. uniformly discrete sequences of uniform lattices) and in \cite{ABBG} for general sequences of manifolds. For simplicity we shall restrict to the special case considered in \cite{7-note}.

Suppose that $G$ is a simple Lie group with $\mathbb{R}$-rank at least $2$ and associated symmetric space $X = G/K $.
Let $X^*$ denote the compact dual of $X$ and recall that the Killing form on $\text{Lie}(G)$ induces a Riammanian structure and a volume form on $X^*$ as well as on $X$.
We denote $\beta _{k} (X)$ the $k^{th}$ $L^{2}$-Betti number of $X$ (see also the analytic definition below).
Then
$$
\beta _{k} (X)= \left\{
\begin{array}{ll}
0, & k\neq \frac{1}{2}\dim X \\
\frac{\chi (X^{\ast })}{\mathrm{vol}(X^{\ast })}, & k=\frac{1}{2}\dim X.%
\end{array} \right.
$$
Note also that $\chi (X^{\ast })=0$ unless $\mathrm{rank}_{\mathbb{C}}(G)=%
\mathrm{rank}_{\mathbb{C}}(K)$.
For an $X$-manifold $M$  we denote by $b_{k}(M)$ the $k^{th}$ Betti number of $M$.

\begin{theorem}
\label{thm:Betti} Let $(M_{n})$ be a sequence of closed $X$-manifolds with
injectivity radius uniformly bounded away from $0$, and $\mathrm{vol}(M_{n})\rightarrow
\infty $. Then:
$$\lim_{n\rightarrow \infty }\frac{b_{k}(M_{n})}{\mathrm{vol}(M_{n})}=\beta
_{k}(X)
$$
for $0\leq k\leq \dim (X)$.
\end{theorem}



\subsection*{The Laplacian and heat kernel on differential forms}



Denote by $e^{-t \Delta_k^{(2)}} (x,y)$ ($x,y\in X$) the heat kernel on $L^2$ $k$-forms. The corresponding bounded integral operator in $\mathrm{End}(\Omega^k_{(2)}(X))$ defined by 
$$
 (e^{-t\Delta_k^{(2)}} f) (x) = \int_{X} e^{-t\Delta_k^{(2)}} (x,y) f(y) \, dy, \ \ \forall f \in  \Omega^k_{(2)}(X)
$$
is the fundamental solution of the heat equation (cf. \cite{BarbaschMoscovici}).

A standard result from local index theory (see e.g. \cite[Lemma 3.8]{BV}) implies:

\begin{lemma}\label{cor:heat}
Let $m >0$. There exists a positive constant $c_1=c_1(G,m)$ such that
$$|| e^{-t \Delta_k^{(2)}} (x,y) || \leq c_1 t^{-d/2} e^{- d (x, y)^2 / 5t } ,  \quad  0 < t \leq m.$$
\end{lemma}

Now let $M=\Gamma\backslash X$ be a compact  $X$-manifold. Let $\Delta_k$ be the Laplacian on differentiable $k$-forms on
$M$. It is a symmetric, positive definite, elliptic operator with a pure point spectrum.
Write $e^{-t\Delta_k} (x,y)$ ($x,y \in M$) for the heat kernel on $k$-forms on $M$, then for each positive $t$ we have:
\begin{equation} \label{eq:heatkernels}
e^{-t\Delta_k} (x,y) = \sum_{\gamma \in \Gamma} (\gamma_{\tilde y})^* e^{-t\Delta_k^{(2)}} (\widetilde{x} , \gamma \widetilde{y}),
\end{equation}
where $\widetilde{x}, \widetilde{y}$ are lifts of $x,y$ to $X$ and by $(\gamma_y)^*$,
we mean pullback by the map $(x,y) \mapsto (\tilde x, \gamma \tilde y)$.
The sum converges absolutely and uniformly for $\tilde{x}, \tilde{y}$ in compacta; this follows from Lemma \ref{cor:heat}, together with the following estimate:
\begin {equation}\label {eq:numtranslates}
 \left|\{ \gamma \in \Gamma \; : \; d(\tilde x,\gamma \tilde y) \leq r \} \right| \leq c_2 e ^ {c_2 r} \,\mathrm{InjRad}_M(x)^{-d} .
\end {equation}
Here, $c_2=c_2(G)$ is some positive constant and $d=\dim(X)$.




\subsection*{($L^2$-)Betti numbers}

\label {bettinumbers}
The trace of the heat kernel $e^{-t \Delta_k^{(2)}}(x,x)$ on the diagonal
is independent of $x \in X$, being $G$-invariant. We denote it by $$\mathrm{Tr} \,e^{-t \Delta_k^{(2)}}:=\mathrm {tr} \, e^{-t \Delta_k^{(2)}}(x,x).$$
The $L_2$-Betti numbers can be defined as:
$$\beta_k (X) :=  \lim_{t \rightarrow \infty} \mathrm{Tr} e^{-t \Delta_k^{(2)}} .$$
By Hodge's theory, the usual Betti numbers of $M$ satisfies
$$
 b_k (M) = \lim_{t \rightarrow \infty} \int_M \mathrm{tr} \ e^{-t \Delta_k}(x,x)dx.
 $$
(Intuitively, the solutions of the heat equations converge to solutions of the Laplace equations when the time goes to infinity --- the temperature stabilizes with time.)


The following is a consequence of Lemma \ref {cor:heat} and Equations (\ref {eq:heatkernels}) and (\ref {eq:numtranslates}).

\begin{cor} \label{lem:hk}
Let $m >0$ be a real number. There exists a constant $c = c(m,G)$ such that for any $x \in M$ and $t\in (0,m]$,
$$
 \left| \mathrm{tr} \ e^{-t\Delta_k } (x,x) - \mathrm{Tr} e^{-t \Delta_k^{(2)}} \right| \leq c \cdot \mathrm{InjRad}_M(x)^{-d}.
$$
\end{cor}

\medskip
\noindent
{\it Proof of Theorem \ref{thm:Betti}.}\label{par:hyp}
Let now $M_n=\Gamma_n\backslash X$ be as in the statement of \ref{thm:Betti}. 
By the result of the previous section (Theorem \ref{thm:7-main}), for every $r>0$ we have
$$
 \lim_{n\to\infty}\frac{\vol ((M_n)_{<r})}{\vol(M_n)}=0,
$$
where $M_{<r}:=\{x\in M \ | \ \mathrm{InjRad}_M(x)<r\}$ is the $r $-thin part of $M $.
Since the injectivity radius is uniformly bounded away from $0$
it follows that
$$
 \frac{1}{\mathrm{vol} (M_n)} \int_{(M_n)_{< r}} \mathrm{InjRad}_{M_n}(x)^{- d} dx \rightarrow 0
$$ 
for every $r>0$.
Hence by Corrolary \ref{lem:hk}:

$$\frac{1}{\vol(M_n)} \int_{M_n} \mathrm{tr} \ e^{-t\Delta_k^{M_n} } (x,x) \, dx \, \longrightarrow \, \mathrm{Tr} e^{-t \Delta_k^{(2)}}$$
uniformly for $t$ on compact subintervals of $(0, \infty)$. Since each term in the limit above is decreasing as a function of $t$, we deduce that

$$
\limsup_{n \rightarrow \infty} \frac{b_k (M_n)}{\mathrm{vol} (M_n)} \leq \beta_k (X).
$$
Finally, using that the usual Euler characteristic is equal to its $L^2$ analog and that $\Delta_k^{(2)}$ has
zero kernel if $k \neq \frac12 \dim X$ we derive Theorem \ref{thm:Betti}.\qed

\begin{rem}
The assumption that $G$ is simple can be replaced by semisimple with property (T) again. 
In \cite{ABBG} the analog of Theorem \ref{thm:Betti} was proved in a much greater generality, removing the assumption that the injectivity radius is bounded below and also allowing the manifolds to be non-compact. This theorem was also extended to the non-Archimedean case in \cite{GL}. In \cite{7A} (as well as \cite{GL}) results concerning the asymptotic of much more general analytic invariants have been established relying on the geometric Theorem \ref{thm:7-main}.
\end{rem}


\section{Stationary random subgroups}\label{sec:SRS}

The notion of invariant random subgroups has proven very effective to the study of certain types of problems. However, the restriction to `invariant' measures limits the scope of problems that one can investigate, in particular as the groups we are studying are non-amenable. For example, the rigidity theorem of Stuck and Zimmer says that for higher rank semisimple Lie groups (with property $(T)$) all IRS comes from lattices. In particular one cannot hope to study thin subgroups or general discrete subgroups which are non-lattices using the theory of IRS. Also for simple groups of rank one, it is known that there are discrete subgroups that cannot be in the support of an IRS. On the other hand, {\it every} discrete subgroup is associated with some stationary random subgroups (hereafter SRS) as we shall explain below.  

\subsection*{First properties}

Let $G$ be a connected semisimple Lie group. Let $\mu$ be a Borel regular probability measure on $G$ whose support generates a Zariski dense subgroup. Let $X$ be a compact $G$-space. A probability measure $\nu\in\text{Prob}(X)$ is called $\mu$-stationary if 
$\mu*\nu=\nu$ where 
$$
 \mu*\nu:=\int g_*\nu d\mu(g).
$$
We let $(B,\nu)$ denote the Poisson boundary of $G$ associated to $\mu$. Then $\nu$ is $\mu$-stationary and has the following universal property. 
Recall that on a compact convex space every probability measure has a unique barycenter. In particular, this applies to probability measures on the  
compact convex space $\text{Prob}(X)$.   

\begin{thm}[Furstenberg]\label{thm:Furstenberg}
Let $X$ be a $G$-space. A probability measure $\nu_0$ on $X$ is $\mu$-stationary iff it is the barycenter of the measure $f_*\nu$ for some $G$-equivariant measurable map $f:B\to\text{Prob}(X)$. 
Moreover, the map $f$ is unique.
\end{thm}

One useful observation that follows from this theorem is that if $\mu_1,\mu_2$ are two probability measures on $G$ which share the same Poisson boundary $(B,\nu)$ then a measure $\nu_0$ on a $G$-space $X$ is $\mu_1$ stationary iff it is $\mu_2$ stationary. 

Let $K$ be a fixed maximal compact subgroup of $G$. Below we will only consider measures $\mu$ which are bi-$K$-invariant. For such measures the Poisson boundary is $(G/P, \nu)$ where $P$ is a minimal parabolic subgroup and $\nu$ is the unique $K$-invariant measure on $G/P$ (the uniqueness is due to the fact that $K$ acts transitively on $G/P$ which is a consequence of the Iwasawa decomposition) . By the observation above we may vary the measure $\mu$ within this class (according to our needs) without affecting the property of being stationary.

\subsection*{A special measure}

Let $G$ be a connected semisimple Lie group without compact factors.
In \cite{GLM} we have constructed a specific measure $\mu_G$ which is quite effective in various situations. 
The measure $M_G$ has the form
$$
 \eta_K*\delta_s*\eta_K
$$
where $\eta_K$ is the Haar measure of the compact group $K$ and $s$ is a certain semisimple element (see \cite[Sections 6,8]{GLM} for details).

\subsection* {The discreteness radius}

Fix a norm $\|\cdot\|$ on the Lie algebra $\text{Lie}(G)$ such that $\text{exp}:\text{Lie}(G)\to G$ restricted to the unit ball $\mathrm{B}(1)=\{X\in\text{Lie}(G):\| X\|<1\}$ is a well defined diffeomorphism. For $r\le 1$ denote $\mathrm{B}(r)=\{X\in\text{Lie}(G):\| X\|<r\}$.
For a discrete group $\gL\subset G$ set
$$
 \mathcal{I}(\gL)=\sup\{r\le 1: \exp \mathrm{B}(r)\cap \Gamma=\{1\}\}.
$$ 
We call $\mathcal{I}(\Gamma)$ the discreteness radius of $\gL$.


\subsection*{The Margulis function on the space of discrete subgroups of $G$}


A fundamental result established in \cite{GLM} is that there is a positive constant $\gd=\gd(G)$ such that 
$$
 u(\gC):=\mathcal{I}(\gC)^{-\gd}
$$
satisfies Inequality (\ref{eq-Marg}) below, that is, it is a Margulis function on the space $\sub_d(G)$ of discrete subgroups of $G$ with respect to $\mu_G$. 

\begin{thm}[\cite{GLM}, Theorem 1.5]\label{thm:eq-Marg}
There exist $0<c<1, b> 0$ such that, for every discrete subgroup $\gC\le G$,
\begin{equation}\label{eq-Marg}
 \int_G u(\gC^g) d\mu_G(g)\leq c u(\gC) +b. 
\end{equation} 
\end{thm}

The constants $\gd, c$ and $b$ are computed explicitly in \cite{GLM} in order to prove certain effective results and in particular a quantitative version of the Kazhdan--Margulis theorem. We will not make use of the explicit values here.

\subsection*{Effective weak uniform discreteness of SRS}

Motivated by \cite{WUD} and the original paper of Kazhdan and Margulis \cite{KM} we have established in \cite{GLM} the following result which is a strong effective version of the Kazhdan--Margulis theorem. It is obviously stronger than the result from \cite{WUD} as it is effective and applies to all SRS rather than IRS. Moreover, since it applies to all SRS it implies the Kazhdan--Margulis theorem for all discrete subgroups.

\begin{theorem}
\label{thm:main theorem}
There are   constants $\beta, \rho, \delta   > 0$   such that for every  
discrete stationary random subgroups $\nu$ on $G$
\begin{equation}
\label{eq:main equation}
 \nu(\{\gL \: : \:  \gL \cap \mathrm{B}(r) \neq \{1   \}) \le  \beta r^\delta \quad \forall r < \rho.
 \end{equation}
\end{theorem}

As a straightforward consequence, we obtain:

\begin{cor}
The space of discrete stationary random subgroups $\text{SRS}_d(G)$ is compact.
\end{cor}

\subsection*{Random walks on the space of discrete subgroups}


The main result of this section is that any stationary limit of a measure supported on discrete subgroups of $G$ is almost surely discrete. 
This is a key ingredient in the proofs of the results of \S \ref{sec:confined} (in particular Theorem \ref{thm:FG1} and Theorem \ref{thm:FG-main}). 

\begin{thm}\label{thm:dSRS}
Let $\nu$ be a probability measure on $\sub_d(G)$. Let 
$$
 \nu_n=\frac{1}{n}\sum_{i=0}^{n-1}\mu_G^{(n)}*\nu
$$ 
and let $\nu_\infty$ be a weak-* limit of $\nu_n$. Then $\nu_\infty$ is supported on the set of discrete subgroups of $G$, that is, $\nu_\infty$ is a discrete stationary random subgroup. 

\end{thm}

The proof is reminiscent of \cite{EM}. 

\begin{proof} 
Recall the discreteness function $\mathcal{I}$ and the Margulis function $u$ defined in the previous page.
By restricting to compact subsets of $\sub_d(G)$, we may allow ourselves to suppose that $\nu$ is compactly supported and that $A:=\int u(\Lambda)d\nu(\Lambda)$ is finite. 

To prove Theorem \ref{thm:dSRS} we need to show that 
\begin{equation}\label{eq-Discrete1}
  \lim_{\gep\to 0}\nu_\infty({\{\Lambda:\mathcal{I}(\Lambda)<\gep\}})= 0.
\end{equation}
Inequality (\ref{eq-Marg}) implies that 
$$
 \int u(\Lambda)d(\mu_G*\nu)(\Lambda)=\int u(\Lambda^g)d\mu_G(g)d\nu(\Lambda)\le cA+b.
$$
Furthermore, iterating Condition (\ref{eq-Marg}) and summing the resulting geometric series we get
\[ 
 \int u(\Lambda)d(\mu_G^{(n)}*\nu)(\Lambda)< c^n A+C,
\] 
with $C:=b/(1-c)$, uniformly for all $n$. Set $M=A+C$.
Then, for every $\gt>0$ and $n\geq 1$ we have
\[ 
 \nu_n(\{\Lambda:u(\Lambda)\geq M\gt^{-1}\})< \gt.
\]
Setting $\gep=(\gt/M)^\frac{1}{\gd}$ we get
\[ 
 \nu_n(\{\Lambda:\mathcal{I}(\Lambda)\le\gep\})<\gt.
\]
Taking $n\to\infty$ gives 
$\nu_\infty(\{\Lambda:\mathcal{I}(\Lambda)\le\gep\})<\gt$,
and letting $\gt \to 0$ we get (\ref{eq-Discrete1}).
\end{proof}


\section{Confined subgroups have finite co-volume}\label{sec:confined}

Let $G$ be a Lie group. A discrete subgroup $\gL\le G$ is called {\it confined} if there is a compact set $C\subset G$ such that $\forall g\in G,~\gL^g\cap C\setminus\{1\}\ne \emptyset$.

\subsection*{Margulis' conjecture}
In this section, we will explain the proof of the following result from \cite{FG} which confirmed a conjecture of Margulis. 

\begin{thm}\label{thm:FG1}
Let $G$ be connected simple Lie group of rank at least $2$. Let $\gL\le G$ be a discrete subgroup. Then $\gL$ is confined iff it is a lattice in $G$.
\end{thm}

Let $X=G/K$ be the associated symmetric space.
Geometrically, the theorem says that if  $M=\gL\backslash X$ is an $X$-orbifold with bounded injectivity radius then it must have finite volume.
In different words, the theorem says that if $\gL\le G$ is a discrete subgroup of infinite co-volume, then it admits a sequence of conjugates $\gL^{g_n}$ which converges to the trivial group. The following statement (from \cite{FG}) is much stronger:

\begin{thm}\label{thm:FG-main}
Let $\mu=\mu_G$ be the probability measure on $G$ constructed in \cite{GLM}. 
Let $\gL\le G$ be a discrete subgroup of infinite co-volume.
Then
$$
 \frac{1}{n}\sum_{i=1}^n \mu^{(i)}*\gd_\gL\to\gd_{\langle 1\rangle}.
$$
\end{thm}

Geometrically, the theorem says that every random walk on a locally $X$-manifold of infinite volume spends most of the time in the $r$-thick part for every $r>0$. 

Let $G$ be as above and $\Lambda\le G$ a discrete subgroup of infinite co-volume. Set 
$$
 \nu_n:=\frac{1}{n}\sum_{i=1}^n \mu^{(i)}*\gd_\gL.
$$ 
The aim is to prove that $\nu_n\to\gd_{\langle 1\rangle}$.
Since $\text{Prob}(\sub(G))$ (with the weak-$*$ topology) is compact, it is enough to prove that $\gd_{\langle 1\rangle}$ is the only limit of $\nu_n$. Thus let $\nu_\infty$ be a week-$*$ limit of $\nu_n$, and we shall argue to show that $\nu_\infty=\gd_{\langle 1\rangle}$. In view of Theorem \ref{thm:dSRS} $\nu_\infty$ is a discrete SRS, i.e. a $\nu_\infty$-random subgroup is almost surely discrete.

\begin{figure}[h]
    \centering
     \includegraphics[width=0.35\textwidth]{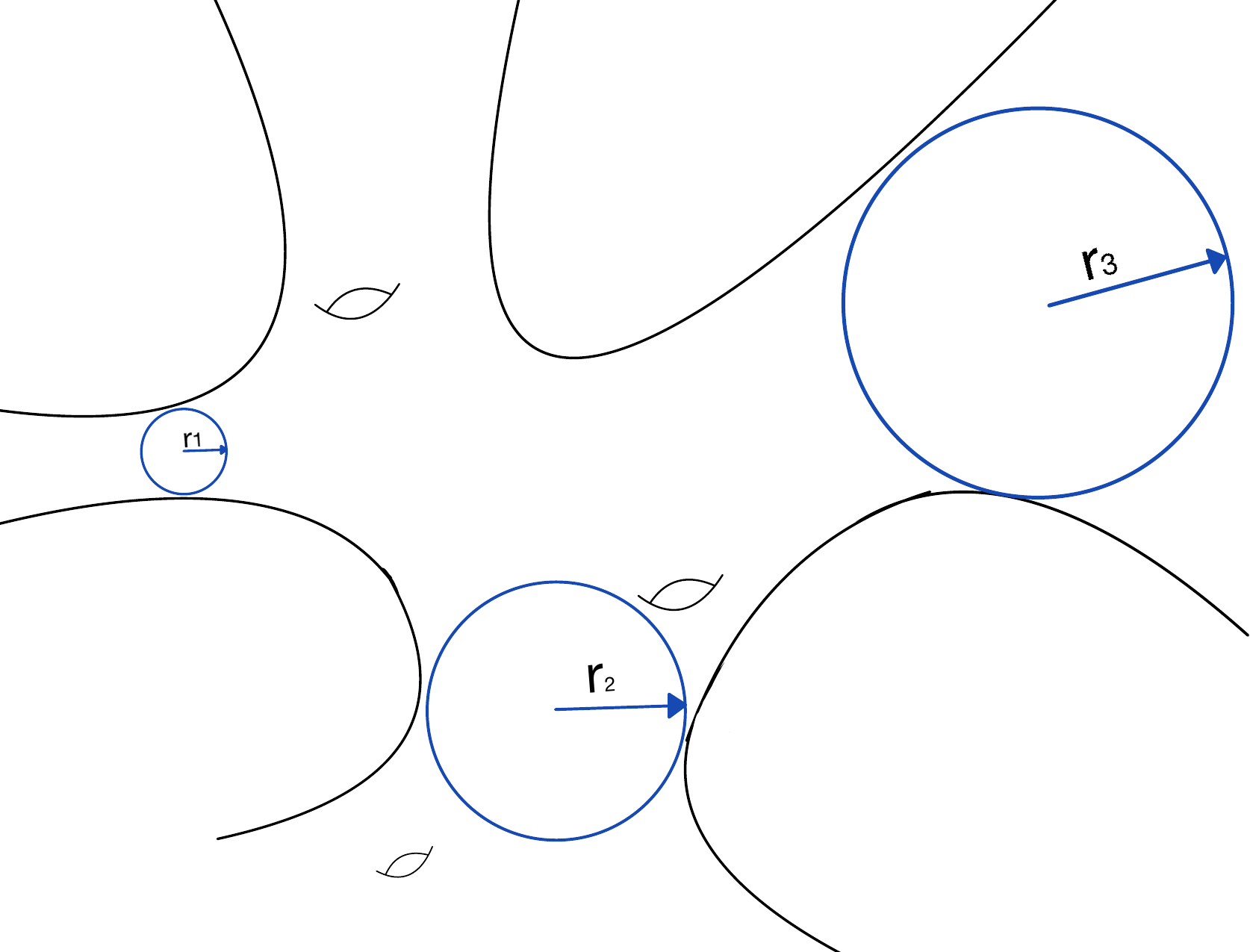}
    \caption{An infinite volume manifold contains contractible balls of any radius.}
    \label{whale}
\end{figure}

\newpage

\subsection*{Discrete stationary random subgroups are invariant.}
We shall start by proving the following stiffness result which is of independent interest:

\begin{thm}[Stiffness]\label{thm:stiffness}
A discrete SRS on $\sub(G)$ is an IRS.
\end{thm}

Fix a minimal parabolic subgroup $P\le G$ and recall that the Poisson boundary of $(G,\mu)$ is $(G/P,\nu_P)$ where $\nu_P$ is the unique $K$-invariant probability measure on $G/P$. Let $\nu_0$ be an SRS. In view of Furstenberg's theorem (see \ref{thm:Furstenberg}) $\nu_0$ is the barycentre of a $G$-equivariant measurable map $\phi:G/P\to \text{Prob}(\sub(G))$, that is:
$$
 \nu_0=\int_{G/P}\phi (gP)d\nu_P.
$$
Moreover, if $\nu_0$ is a discrete SRS, i.e. if $\nu_0(\sub_d(G))=1$ then for $\nu_P$-almost every $gP\in G/P$, $\phi(gP)$ is a discrete random subgroup as well. 

Theorem \ref{thm:stiffness} is a consequence of the following celebrated theorem of Nevo and Zimmer \cite{NZ1,NZ3}:

\begin{thm}[Nevo--Zimmer]\label{thm-NZ}
Let $G$ be a higher rank simple Lie group. Let $\mu$ be a smooth probability measure on $G$ and let $(X,\nu)$ be a probability $\mu$-stationary action of $G$. Then either
$\nu$ is $G$-invariant, or there exists a proper parabolic subgroup $Q\subset G$ and a measure preserving $G$-equivariant map\footnote{ By measure preserving we only mean that $\nu(\pi^{-1}(A))=\mu_Q(A)$ for every Borel set $A\subset G/Q$. In particular, we do not require that it is a relative measure preserving extension.} $\pi\colon (X,\nu)\to (G/Q,\mu_Q),$ where $\mu_Q$ is the unique $\mu$-stationary measure on $G/Q$.
\end{thm}

Our specific measure $\mu=\mu_G$ is not smooth. However, in view of the remark after Theorem \ref{thm:Furstenberg}, since there exists some smooth bi-$K$-invariant probability measure on $G$, Theorem \ref{thm-NZ} applies also for $\mu_G$.

Arguing by way of contradiction, suppose that $\nu_0$ is a non-invariant discrete SRS. It follows from Theorem \ref{thm-NZ} that there is a parabolic subgroup $Q$ containing $P$ such that for $\nu_P$ almost all $gP\in G/P$ the measure $\phi(gP)$ is a $gPg^{-1}$-invariant discrete random subgroup of $gQg^{-1}$. The contradiction, thus follows from the following:

\begin{prop}\label{prop:P,Q}
Let $P\le Q$ be a parabolic subgroup of $G$. The only $P$-invariant discrete random subgroup of $Q$ is $\gd_{\langle 1\rangle}$.
\end{prop}

\begin{lemma}\label{lem-AIRS} Let $Q$ be a parabolic subgroup of $G$. Let $Q=LN$ be a Levi decomposition of $Q$ and let $A_L$ be the center of $L$. Then, any discrete $A_L$-invariant random subgroup of $Q$ is contained in $L$ almost surely. 
\end{lemma}

\begin{proof}

Let $a\in A_L$ be an element such that the norm of the restriction of $\text{Ad}(a^{-1})$ to the Lie algebra of $N$ is less than $1$. It follows that for any two identity neighbourhoods $U_1,U_2\subset N$ in $N$ such that $U_2$ is bounded there is $n_0\in\BN$ such that $\text{Ad}(a^{-n_0})(U_2)\subset U_1$.

Let $\lambda$ be a discrete $A_L$-invariant random subgroup of $Q$. Suppose by way of contradiction that $\lambda (\{\gC\subset L\})<1$. Then there is a bounded identity neighborhood $V_2$ in $Q$ such that 
$\lambda(\gO)>0$ where $\gO:=\{\gC:\gC\cap V_2\setminus L\ne\emptyset\}$.
Furthermore, we may take $V_2$ to be of the form $V_2=W\cdot U_2$ where $W\subset L$ and $U_2\subset N$ and suppose that they are preserved by $a^{-1}$ (e.g. we can suppose that $\log U_2$ is a norm ball in the Lie algebra of $N$). Since $\lambda$ is discrete we can chose a small identity neighborhood $U_1\subset N$ so that setting $V_1=W\cdot U_1$ we have 
$$
 \gep:=\lambda(\{\gC\in\gO:\gC\cap (V_1\setminus L)=\emptyset \})>0.
$$
Choosing $n_0$ as above we get that $\lambda (\gO^{a^{n_0}})\le\lambda(\gO)-\gep$ in contrast with the assumption that $\lambda$ is $A_L$-invariant.
\end{proof}

In view of Lemma \ref{lem-AIRS}, Proposition \ref{prop:P,Q} follows from the following Lie theoretic result (see \cite[Lemma 6.3]{FG}):

\begin{lemma}\label{lem-LeviInt}
The intersection of all Levi subgroups of $Q$ is trivial.
\end{lemma}

The contradiction is now achieved, i.e. assuming that $\nu_0$ is non-invariant we deduce that $\phi(gP)=\gd_{\langle 1\rangle}$ and hence $\nu_0=\int_{G/P}\gd_{\langle 1\rangle}d\nu_P=\gd_{\langle 1\rangle}$ which is absurd. 
This completes the proof of Theorem \ref{thm:stiffness}. \qed 

\begin{rem}
We remark that the analog of stiffness Theorem \ref{thm:stiffness} does not hold for rank one groups. We refer to \cite[Example 6.8]{FG} for an example of a discrete SRS in $\SL_2(\BR)$ which is not an IRS. Relying on \cite{Osin} one can construct examples of that nature in any Gromov hyperbolic group and hence in 
any rank one simple Lie group.
\end{rem}

\subsection*{The proof of Theorem \ref{thm:FG-main}.}

Recall that we fixed a discrete subgroup of infinite co-volume  $\gL\le G$ and denoted $\nu_n:=\frac{1}{n}\sum_{i=1}^n \mu^{(i)}*\gd_\gL$ and $\nu_\infty$ an arbitrary accumulation point of $\nu_n$. We will prove Theorem \ref{thm:FG-main} by showing that $\nu_\infty$ must be $\gd_{\langle 1\rangle}$. Since $\nu_\infty$ is stationary, it follows from Theorem \ref{thm:stiffness} that it is invariant, i.e. an IRS. Moreover, in view of Theorem \ref{thm:dSRS}, $\nu_{\infty}$ must be discrete. By the Stuck--Zimmer theorem (Theorem \ref{thm:SZ}) any non-trivial subgroup in the support of $\nu_\infty$ must be a lattice. Therefore in order to prove that $\nu_{\infty}=\gd_{\langle 1\rangle}$ it is enough to show that there are no lattices in its support. Recall that the $\nu_n$ are supported on the conjugacy class of $\gL$. Therefore $\nu_\infty$ is supported on the closure of $\gL^G$. It is therefore enough to show that the closure of the conjugacy class $\overline{\gL^G}$ cannot contain lattices. In order to complete the proof we shall refer to the classical local rigidity theorem.

\paragraph{Local rigidity.} Let $\gC$ be a finitely generated group. Consider the deformation space $\text{Hom}(\gC,G)$ of homeomorphisms from $\gC$ to $G$ with the point-wise convergence topology. A map $\psi\in \text{Hom}(\gC,G)$ is called locally rigid if there is a neighborhood of $\psi$ which consists of conjugates. That is if $\phi\in \text{Hom}(\gC,G)$ is sufficiently close to $\psi$ then there is $g\in G$ such that $\phi=\psi^g$. A subgroup $\gC$ of $G$ is called {\it locally rigid} if the inclusion $i:\gC\hookrightarrow G$ is locally rigid.
The classical local rigidity theorem due to Selbeg, Weil, and Margulis, says the following:

\begin{thm}
Let $G$ be a non-compact simple Lie group not locally isomorphic to $\SL_2(\BR)$ or $\SL_2(\BC)$. Then every lattice in $G$ is locally rigid. 
\end{thm}

We can now proceed with the proof of Theorem \ref{thm:FG-main}.
Suppose by way of contradiction that there is a sequence of conjugates $\gL^{g_n}$ which converges to a lattice $\gC\le G$. Since $\gC$ is discrete there is a compact identity neighborhood $U\subset G$ with $U\cap \gC=\{1\}$. Recall that $\gC$, being a lattice is finitely presented (see \cite{HV,crells}). Let $\langle \gS:R\rangle$ be a finite presentation of $\gC$. For $n$ large enough we have both that
\begin{itemize}
\item $\gL^{g_n}\cap U=\{1\}$, and
\item if $\gs_n\in\gL^{g_n}$ is a nearest element to $\gs$ for any $\gs\in \gS$ then for every $w\in R$, the evaluation of $w$ on $\gs_n,~\gs\in\gS$ falls in $U$.
\end{itemize}
It follows that all the relations $w\in R$ are satisfied on the set $\{\gs_n:\gs\in \gS\}$ and hence the assignment $\gs\mapsto \gs_n$ a homomorphism of $\gC$ which is close to the inclusion. By local rigidity, we deduce that $\gL^{g_n}$ contains a conjugate of $\gC$. It follows that $\gL$ is a lattice, in contrast to the assumption. This completes the proof of  Theorem \ref{thm:FG-main}. \qed

\paragraph{Summary of the proof:} Theorem \ref{thm:FG-main} is proven in three steps:
\begin{enumerate}
\item Discreteness --- Theorem \ref{thm:dSRS}.
\item Stiffness --- Theorem \ref{thm:stiffness}.
\item Rigidity --- the use of Stuck--Zimmer rigidity theorem and local rigidity, as described above.
\end{enumerate}

\subsection*{The case of semisimple groups}

In \cite{FG} several analogs of Theorem \ref{thm:FG-main} were proved also for general semisimple groups. This relied on the following decomposition theorems for discrete IRS and SRS (see \cite[Theorem 4.1, Corollary 4.4 and Theorem 6.5]{FG}):

\begin{thm}[\cite{FG}, Theorem 4.1 and Corollary 4.4]\label{thm:decomposition}
Let $G=G_1\times\cdots\times G_n$ be a connected center-free semisimple Lie group without compact factors and with simple factors $G_i,~i=1,\ldots,n$.
Let $\nu$ be an ergodic discrete invariant random subgroup in $G$. Then $G$ decomposes to a product of semisimple factors $G=H_1\times\ldots\times H_k$ with $1\le k\le n$, such that almost surely the projection of a random subgroup to each $H_i$ is discrete while the projection of the intersection with $H_i$ to each proper factor of $H_i$ is dense. 
\end{thm}

\begin{thm}[\cite{FG}, Theorem 6.5]\label{thm:stationary-decomposition}
Let $G$ be a connected centre-free semisimple Lie group without compact factors and $\nu$ a discrete $\mu$-stationary random subgroup of $G$. Then $G$ decomposes to a product of three semisimple factors $G=G_\mathcal{I}\times G_\mathcal{H}\times G_\mathcal{T}$ such that
\begin{enumerate}
\item $\nu$ projects to an IRS in $G_\mathcal{I}$ for which all the irreducible factors are of rank at least $2$.
\item $G_\mathcal{H}$ is a product of rank one factors and $\nu$ projects discretely to every factor of $G_\mathcal{H}$.
\item $\nu$ projects trivially to $G_\mathcal{T}$.
\end{enumerate}
Furthermore, the intersection of a random subgroup with every simple factor of $G_\mathcal{H}$ as well as with every irreducible factor of $G_\mathcal{I}$ is  almost surely Zariski dense in that factor.
\end{thm}

By the irreducible factors of $G_\mathcal{I}$ we mean the irreducible factors associated with the decomposition of IRS as in Theorem \ref{thm:decomposition}. The subscripts $\mathcal{I}, \mathcal{H},\mathcal{T}$ stands for invariant, hyperbolic and trivial (respectively). 

Theorem \ref{thm:stationary-decomposition} allows the extension of Theorem \ref{thm:FG-main} to semisimple groups under certain conditions. The simplest case is when all the factors of $G$ are of rank at least $2$:

\begin{thm}
Let $G=G_1\times\cdots\times G_n$ be a connected center-free semisimple Lie group such that each $G_i$ is simple of rank at least $2$.
Let $\Lambda$ be a discrete subgroup. Then $\Lambda$ is confined if and only if there is a nontrivial semisimple factor $H\lhd G$ such that $\Lambda\cap H$ is a lattice in $H$. If $\Lambda$ is not confined then 
 $\frac{1}{n}\sum_{i=0}^{n-1}\mu_G^{(i)}*\gd_{\gL}\to\gd_{\{1\}}$.
\end{thm}

We refer to \cite[Section 9]{FG} for additional generalizations of Theorem \ref{thm:FG-main} which allows also rank one factors. 
We remark that Theorem \ref{thm:stationary-decomposition} plays also an important role in the results described in the next section.

\section{Spectral gap for products and a strong version of the NST}\label{sec:spectral-gap}

\subsection*{Margulis' normal subgroup theorem}

Recall the celebrated normal subgroup theorem of Margulis (see \cite{Ma} and \cite{Ma1}):

\begin{thm}\label{thm:NST}
Let $G$ be a center-free semisimple Lie group without compact factors with $\text{rank}(G)\ge 2$, and let $\gC\le G$ be an irreducible lattice. Then every non-trivial normal subgroup $N\lhd \gC$ is of finite index. 
\end{thm}

The impact of this theorem on the field is enormous. The proof is quite remarkable, both in the general idea and in the details. Margulis deduces the finiteness of $\gC/N$ from the tension between two analytic properties that cannot get along with each other for infinite groups, namely amenability and property $(T)$.  

Recall that a group $H$ is amenable if the regular representation on $L^2(H)$ admits asymptotically invariant vectors, while property $(T)$ is a spectral gap with respect to all unitary representations and in particular the regular representation. Therefore if $H$ is both amenable and has $(T)$ then $L^2(H)$ must admit the constant non-zero functions which implies that $H$ is compact, or finite if it is also discrete.

In case the semisimple group $G$ has property $(T)$ one gets for free that also $\gC$ and $\gC/N$ have $(T)$. Hence in that case Margulis' proof boils down to showing that $\gC/N$ is amenable. The case where $G$ does not have property $(T)$, e.g. for $G=\PSL_2(\BR)\times\PSL_2(\BR)$ requires additional (and again quite remarkable) work and is carried out in \cite{Ma1}.

There are several works that generalize and improve the normal subgroup theorem. For instance Thereom \ref{thm:SZ} can be regarded as a measure-theoretic counterpart of the NST, while Theorem \ref{thm:FG1} can be regarded as a geometric variant of the NST and both are much stronger than the NST when they hold. However these works, as well as all other published generalizations of the NST for classical semisimple groups relied on Kazhdan's property (T) and do not cover the general higher rank case.

\subsection*{Higher rank irreducibly confined subgroups are lattices}

I will now review a very recent
joint work with Arie Levit and Uri Bader \cite{BGL}.
We have managed to prove the following:

\begin{thm}[\cite{BGL}]\label{thm:BGL1}
Let $G$ and $\gC$ be as in Theorem \ref{thm:NST}. Then every confined subgroup of $\gC$ is of finite index.
\end{thm}

In other words, infinite index subgroups of higher rank irreducible lattices are unconfined. That is, if $|\gC:\gL|=\infty$ then there is a sequence $\gc_n\in\gC$ such that $\gL^{\gc_n}\to\langle 1\rangle$. This property is much stronger than being non-normal (assuming $\gL$ is nontrivial). Indeed, normal subgroups are the fixed points for the action $\gC\act\sub(\gC)$.

When $G$ (or at least one of its simple factors) has property $(T)$, Theorem \ref{thm:BGL1} follows from the much stronger results proved in \cite{FG} (see \S \ref{sec:confined}).
The novelty in our new work is that it applies regardless of property $(T)$. 

As in \cite{FG}, we also obtained a general result for discrete groups which are not assumed a priory to be contained in a lattice. For this result we need however to impose some irreducibility condition which passes to limits of conjugates:

\begin{defn}
Let $G$ be a semisimple Lie group. A discrete subgroup of $G$ is called {\it strongly confined} if it has no discrete conjugate limit which is contained in a proper factor of $G$.
A discrete subgroup is called {\it irreducibly confined} if it is strongly confined and its intersection with any proper factor of $G$ is not confined.
\end{defn}

\begin{thm}[\cite{BGL}]\label{thm:BGL2}
Let $G$ be a connected center-free semisimple Lie group without compact factors and with $\text{rank}(G)\ge 2$. A discrete subgroup $\gL\le G$ is irreducibly confined in $G$ iff it is an irreducible lattice in $G$.
\end{thm}

The paper \cite{BGL} is self contained, hence can also serve as an alternative to \cite{Ma1} for a proof of the NST in this most complicated case.\footnote{Note that the proof of the NST in the lack of property (T) is not fully contained in \cite{Ma}.}

\subsection*{Hints about the proof}

Recall the following lemma \cite[Lemma 7.2]{FG}:

\begin{lemma}\label{lem:G1dense->G2-Invartant} 
Let $G=G_1\times G_2$ where $G_1,G_2$ are locally compact second countable groups. Let $X$ be a compact, second countable $G$-space with $G_1$-invariant probability measure $\nu$ such that ${\rm Stab}_G (x)$ is discrete and has a dense projection onto $G_2$ for $\nu$-almost every $x\in X$. Then $\nu$ is $G$-invariant. 
\end{lemma}

\begin{proof}
In view of the ergodic decomposition of probability measure preserving actions, the lemma clearly reduces to the ergodic case, so let us assume that $\nu$ is $G_1$-ergodic. Using the large stabilizers of the action and Kakutani's ergodic theorem for random walks we will prove that such a measure must be also $G_2$-invariant.

Choose a smooth symmetric probability measure $\eta$ on $G_1$ whose support generates $G_1$. Write $\eta^{g_1}$ for the measure $\eta^{g_1}(A)=\eta(g_1^{-1}A g_1)$. Since the support of $\eta^{g_1}$ generates $G_1$, the measure $\nu$ is $\nu^{g_1}$ stationary and ergodic with respect to the random walk on $X$ induced by $\eta^{g_1}$, for every $g_1\in G_1$. We say that a point $x\in X$ is $\eta^{g_1}$-generic for $\nu$ if
\begin{equation}\label{eq-Birkhoff}
\lim_{n\to\infty}\frac{1}{n} \sum_{i=1}^{n-1} \int f(g x)d(\eta^{g_1})^{\ast i}(g)=\lim_{n\to\infty}\frac{1}{n} \sum_{i=1}^{n-1} \int f(g_1^{-1}g g_1 x)d\eta^{\ast i}(g)=\int f d\nu(x),
\end{equation} for all continuous functions $f$. By Kakutani's pointwise ergodic theorem for random walks \cite{Kak51}, the set of $\eta^{g_1}$-generic points has full measure with respect to $\nu$. It follows that the set \[\{ (g_1,x)| x \textrm{ is not } \eta^{g_1} \textrm{-generic}\}\subset G_1\times X\] has zero measure. By Fubini's theorem we deduce that there is a subset $X'\subset X$ with $\nu(X')=1$, such that every $x\in X'$ satisfies the following property. For almost every $g_1\in G_1$  
\begin{equation}\label{eq-ErgAv}
\lim_{n\to\infty}\frac{1}{n} \sum_{i=1}^{n-1}  (\eta^{g_1})^{\ast i}\delta_{x}=\nu.
\end{equation} in the weak-* topology. 

Take a point $x\in X'$ such that $\Gamma:={\rm Stab}_G(x)$ is discrete and has a dense projection onto $G_2$. Since $\Gamma$ is countable, we may fix  $g_1\in G$ such that (\ref{eq-ErgAv}) holds for $g_1\gamma_1$, for every $\gamma=(\gamma_1,\gamma_2)\in \Gamma$. Note that $\eta^{\gamma g}=\eta^{\gamma_1g_1}$. Comparing (\ref{eq-ErgAv}) for $g_1$ and $\gamma_1 g_1$ and using the fact that $\gamma^{-1}x=x$ we find that \begin{align*}\nu=&\lim_{n\to\infty}\frac{1}{n} \sum_{i=1}^{n-1}  (\eta^{\gamma_1g_1})^{\ast i}\delta_{x}=\lim_{n\to\infty}\frac{1}{n} \sum_{i=1}^{n-1}  (\eta^{\gamma g_1})^{\ast i}\delta_{x}\\=&\lim_{n\to\infty}\frac{1}{n} \sum_{i=1}^{n-1}  \gamma_* (\eta^{g_1})^{\ast i}\delta_{\gamma x}=\gamma_*\lim_{n\to\infty}\frac{1}{n} \sum_{i=1}^{n-1}  (\eta^{g_1})^{\ast i}\delta_{x}=\gamma_*\nu.\end{align*} This means that $\Gamma G_1\subset {\rm Stab}_G \nu$. The action of $G$ on the set of probability measures on $X$ is continuous, so ${\rm Stab}_G \nu\supset \overline{\Gamma G_1}=G.$
\end{proof}

In a sense, Theorem \ref{thm:BGL2} is established by proving a careful suitable approximated version of Lemma \ref{lem:G1dense->G2-Invartant}. That is, we establish that certain measures which are sufficiently $G_1$ invariant must also be $G_2$ almost invariant. However, the precise formulation and the proof are considerably more complicated and we refer the reader to \cite{BGL}. Below we shall outline how data of this type implies the desired statements.

Since simple groups with rank at least $2$ have property $(T)$, Theorem \ref{thm:BGL2} (as well as Theorem \ref{thm:BGL1}) follows from \cite{FG} in that case. Therefore we may suppose that $G$ has more than one factor. Let $\gL\le G$ be an irreducibly confined subgroup. Throughout most of the proof, we suppose that $\gL$ is also co-amenable. That is, there is an asymptotically invariant sequence of unit vectors $f_n\in L_2(G/\gL)$.  This means that if $\Omega$ is a compact set in $G$ then 
$$
 \sup_{g\in\gO} \| L_g(f_n)-f_n\| \to 0,
$$
Where $L_g(f)(x):=f(g^{-1}x)$.

Since we suppose that $G$ is not simple we can write it as $G_1\times G_2$. We then proceed in a similar fashion to the argument in \cite{Ma1} and show that we can modify the sequence $f_n$ to another sequence $h_n$ which is asymptotically invariant with respect to $G_1$ but uniformly non-invariant with respect to $G_2$. That means that for some fixed compact generating sets $\gO_i$ for $G_i,~i=1,2$ we have that 
$$
 \sup_{g\in\gO_1} \| L_g(h_n)-h_n\| \to 0~\text{while}~\sup_{g\in\gO_2} \| L_g(h_n)-h_n\|\ge\gep 
$$
for some fixed $\gep>0$. This part of our argument is similar to the proof from \cite{Ma1} but we do it in a much more general setup, and in a more direct way. 

We then consider $h_n^2$ as functions in $L_1(G/\gL)$ and as such we view them as distributions of probability measures. Then we consider the map
$$
 G/\gL\to\sub (G),~g\gL\mapsto g\gL g^{-1},
$$
and push these measures to probability measures on $\sub (G)$ obtaining a sequence of discrete random subgroups which are asymptotically $G_1$-invariant but uniformly non-invariant under $G_2$. We then derive a contradiction by a certain approximated version of Lemma \ref{lem:G1dense->G2-Invartant}. The proof however is considerably more delicate. 

\begin{rem}
The above sketch is an oversimplification of the actual proof avoiding various nuances. For instance, it is crucial for us the the sequence $h_n$ is {\it uniform} in the sense that 
$$
 \lim_{g\to 1}\sup_n \{  \| L_g(h_n)-h_n\|\}=0.
$$
\end{rem}

\subsection*{From general subgroups to co-amenable ones}

We sketched in the previous subsection some ideas of the proof of Theorems \ref{thm:BGL1} and \ref{thm:BGL2} under the additional assumption that the irreducibly confined subgroup $\gL\le G$ is also co-amenable in $G$. Let us now explain how to remove the co-amenability assumption.

Supposing we already know that irreducibly confined co-amenable subgroups are lattices, we wish to show that every irreducibly confined subgroup is a lattice.
Let $\gL\le G$ be an irreducibly confined subgroup. Consider the ergodic averages of the $\mu=\mu_G$ random walk on $\sub_d(G)$ starting at $\gd_\gL$:
$$
 \frac{1}{n}\sum_{i=1}^n\mu^{(i)}*\gd_\gL
$$ 
and let $\nu$ be an arbitrary weak-* limit. Then $\nu$ is $\mu$-stationary and by Theorem \ref{thm:dSRS} it is a discrete SRS. Applying the decomposition Theorem
\ref{thm:stationary-decomposition}, keeping in mind that $\nu$ is supported on conjugate limits of $\gL$ and hence on irreducibly confined groups, we get that $G=G_\mathcal{I}$. That is $\nu$ is an irreducible IRS. In view of Theorem \ref{thm:SZ,HT}, a $\nu$-random subgroup is almost surely 
co-amenable. It follows that $\gL$ admits a co-amenable conjugate limit $\gC$. By what we already agreed on $\gC$ must be an irreducible lattice in $G$. In particular $\gC$ is locally rigid. This implies that $\gL$ admits a subgroup conjugated to $\gC$, hence $\gL$ is also a lattice.
\qed

\subsection*{Spectral gap for products}

The presentation of the previous two subsections may make the impression that the `irreducibly confined' assumption is only required in the reduction to the co-amenable case and that for co-amenable groups a weaker irreducibility assumption might be sufficient. However, this is not the case. 
Even in the co-amenable case, the `irreducibly confined' assumption is crucial in our argument. The reason is that we apply a process of changing the probability measures and we have limited control on the way this is done. Therefore we need to assume some irreducibility also for Chabauty limits of stabilizers. The spectral gap result that we obtain in this case is quite general, however:

\begin{theorem}[Spectral gap for actions of products, \cite{BGL}]
\label{theorem:getting spectral gap}
Let $G_1$ and $G_2$ be a pair of locally compact second countable compactly generated groups such that $G_2$ has compact abelianization. Set $G = G_1 \times G_2$. 
Let $X$ be a locally compact topological $G$-space endowed with a $G$-invariant (finite or infinite) measure $m$. Assume that
\begin{itemize}
    \item $L^2_0(X,m)^{G_2}=0$, and
    \item There is a $G$-invariant closed subset of $\Sub(G)$ containing $\mathrm{Stab}_G(x)$ for $m$-almost every  point $x\in X$ such that every subgroup $H$ in  this subset satisfies $\overline{G_1 H} = G$.
\end{itemize}
Then the Koopman $G$-representation $L^2_0(X,m)$ has a spectral gap.
\end{theorem}

For semismmple groups we obtain the following.

\begin{theorem}[Spectral gap for actions of products --- semisimple group case, \cite{BGL}]
\label{theorem:getting spectral gap - analytic groups}
Let $G$ be an adjoint semisimple group without compact factors. Write $G = G_1 \times G_2$ where $G_1$ is semisimple and $G_2$ is \emph{simple}. Let $X$ be a locally compact topological $G$-space endowed with a $G$-invariant measure $m$, either finite or infinite. Assume that
\begin{itemize}
\item $L^2_0(X,m)^{G_2} = 0$,
\item  $m$-almost every point $x$ has $\mathrm{Stab}(x) \cap G_1 = \mathrm{Stab}(x) \cap G_2 = \{e\}$ and
\item if $f_n \in L^2(X,m)$ is an asymptotically $G_1$-invariant sequence of unit vectors then every  accumulation point  $\mu \in \text{Prob}(\Sub G)$ of the sequence of probability measures $\mathrm{Stab}_*(|f_n|^2 \cdot m)$ is such that $\mu$-almost every subgroup is  discrete,  not contained in the factor $G_2$  and admits a Zariski-dense projection to $G_2$.
\end{itemize}
Then the Koopman $G$-representation $L^2_0(X,m)$ has a spectral gap.
\end{theorem}

\end{document}